\documentclass[12pt, a4paper]{amsart}

\usepackage[hmargin=32mm, vmargin=27mm, includefoot, twoside]{geometry}
\usepackage[bookmarksopen=true,colorlinks,linkcolor={red},citecolor={blue}]{hyperref}
\usepackage{graphicx,tikz}
\usepackage[english]{babel}
\usepackage{amsmath,amssymb}
\usepackage{lmodern}
\usepackage[figurename=Fig.]{caption}
\usepackage{color}

\newtheorem{thmintro}{Theorem}

\newtheorem{theorem}{Theorem}[section]
\newtheorem{corollary}[theorem]{Corollary}
\newtheorem{lemma}[theorem]{Lemma}
\newtheorem{prop}[theorem]{Proposition}
\theoremstyle{definition}

\newtheorem{example}[theorem]{Example}
\newtheorem{definition}[theorem]{Definition}
\newtheorem{claim}[theorem]{Claim}

\newcommand{\NN}{\mathbb{N}}
\newcommand{\N}{\mathbb{N}}
\newcommand{\RR}{\mathbb{R}}
\newcommand{\CCC}{\mathcal{C}}
\newcommand{\FFF}{\mathcal{F}}
\newcommand{\co}{\colon\thinspace}

\DeclareMathOperator{\CGR}{CGR}
\DeclareMathOperator{\dist}{d}
\DeclareMathOperator{\Geo}{Geo}
\newcommand{\dc}{\dist}
\DeclareMathOperator{\Chor}{\CCC_{horo}}
\DeclareMathOperator{\Chb}{\CCC_{hb}}
\DeclareMathOperator{\Xs}{X_{s,\eta}^{(0)}}
\DeclareMathOperator{\im}{im}
\DeclareMathOperator{\proj}{proj}

\begin{document}

\renewcommand{\proofname}{{\bf Proof}}

\title{Hyperfiniteness of boundary actions of hyperbolic groups}

\author[T.~Marquis]{Timoth\'ee \textsc{Marquis}$^{(*)}$}
\address{UCLouvain, IRMP-MATH, Chemin du Cyclotron 2, 1348
  Louvain-la-Neuve, Belgium}
\email{timothee.marquis@uclouvain.be}
\thanks{$^{(*)}$Corresponding author (ORCID 0000-0003-0541-8302, timothee.marquis@uclouvain.be); F.R.S.-FNRS Postdoctoral Researcher}
\author[M.~Sabok]{Marcin \textsc{Sabok}$^\dagger$}
\address{McGill University, Department of Mathematics and
  Statistics, 805 Sherbrooke Street W, Montreal, QC, H3A 0B9
  Canada and Instytut Matematyczny PAN, Sniadeckich 8,
  00-656 Warszawa, Poland} \email{marcin.sabok@mcgill.ca}
\thanks{$^\dagger$ This research was partially supported by
  the NSERC through the \textit{Discovery Grant}
  RGPIN-2015-03738, by the FRQNT (Fonds de recherche du
  Qu\'{e}bec) grant \textit{Nouveaux chercheurs}
  2018-NC-205427 and by the NCN (National Science
  Centre, Poland) through the grants \textit{Harmonia}
  no. 2015/18/M/ST1/00050 and 2018/30/M/ST1/00668}

\begin{abstract}
  We prove that for every finitely generated hyperbolic group $G$, the action of $G$ on its Gromov boundary induces a hyperfinite equivalence relation.
\end{abstract}

\maketitle

\section{Introduction}
\label{sec:introduction}
The complexity theory of Borel equivalence relations has
been developed in the last forty years as an attempt to measure
the difficulty of classification problems (\cite{Kan08}). A particularly
active part of the theory is concerned with the structure
of countable Borel equivalence relations. As all countable
Borel equivalence relations are induced by Borel actions of
countable groups, this has a close relationship with the study of
Borel or measurable group actions on standard Borel spaces.

The interplay between the structure of measurable actions of
countable (or finitely generated) groups and the algebraic
properties of the groups is one of the main themes appearing
at the intersection of group theory, ergodic theory and
descriptive set theory. The classical result \cite[Theorem 7.2.4]{gao} of
Slaman--Steel (\cite{SS88}) and Weiss (\cite{Wei84}) characterises the
equivalence relations induced by the actions of $\mathbb{Z}$
(i.e. by a single automorphism) as the \textbf{hyperfinite}
equivalence relations: those which can be written as an
increasing union of finite equivalence relations. This
notion has been studied both from the Borel and measurable
point of view.

The above characterisation holds in a pure Borel context,
where there is no probability measure around. Given a Borel
probability measure $\mu$, an equivalence relation is
\textbf{$\mu$-hyperfinite} if it is hyperfinite restricted
to a certain subset of measure $1$.  In the presence of an invariant probability measure $\mu$, every amenable group
action induces a $\mu$-hyperfinite equivalence relation. It is one of the notorious open problems whether
the latter holds in the pure Borel setting.

Hyperbolic groups and spaces were introduced and studied by
Gromov (see \cite{gromov}) and have attracted a lot of attention in geometric
group theory. To every (geodesic) proper hyperbolic metric space $X$ one associates a compact metric space $\partial X$, called its \textbf{Gromov boundary}, which is a quasi-isometry invariant of $X$. One then also defines the Gromov boundary $\partial G$ of a finitely generated hyperbolic group $G$ as the Gromov boundary of its Cayley graph (with respect to a finite generating set), and the $G$-action on its Cayley graph naturally induces a $G$-action by homeomorphisms on $\partial G$. 

Boundary actions of hyperbolic groups have also been studied from
the point of view of their complexity. In the case of the free group $G$,
Connes, Feldman and Weiss \cite[Corollary~13]{cfw} and Vershik
\cite{vershik} showed that the action of $G$ on its Gromov boundary $\partial G$ is
$\mu$-hyperfinite for every Borel quasi-invariant
probability measure $\mu$ on $\partial G$. This was later generalised by Adams in \cite{adams} to all finitely generated hyperbolic groups.

On the other hand, in \cite[Corollary~8.2]{djk} Dougherty, Jackson and
Kechris proved that the boundary action of the free group $F_2$ is
hyperfinite (in the pure Borel sense) by studying the
so-called \textbf{tail equivalence relation}. More recently,
Huang, Sabok and Shinko showed in \cite{HSS17} that every cubulated
hyperbolic group $G$ (i.e. $G$ acts geometrically on a
hyperbolic CAT(0) cube complex $X$) has a hyperfinite boundary
action. Their proof is based on an analysis of geodesic ray
bundles in hyperbolic CAT(0) cube complexes. More precisely,
given a point $\eta\in\partial X$ and a vertex $x$ of the
1-skeleton $X^{(0)}$ of $X$, the \textbf{geodesic ray bundle}
$\Geo(x,\eta)$ consists of all vertices that appear on a
geodesic ray starting at $x$ and converging to $\eta$. The hyperfiniteness of the $G$-action on $\partial X$ (or equivalently, on $\partial G$ since $X$ and $G$ are quasi-isometric) is then established as a consequence of the following geometric condition (see \cite[Theorem~1.4]{HSS17}): for every $x,y\in X^{(0)}$ and $\eta\in\partial X$, the sets $\Geo(x,\eta)$ and $\Geo(y,\eta)$ have a finite symmetric difference.

In \cite[Question~1.5]{HSS17} it was asked whether this condition holds in Cayley graphs of arbitrary
hyperbolic groups. However, Touikan constructed in \cite{touikan}
hyperbolic groups (or rather, appropriate sets of generators
for the free group) with an associated Cayley graph in which this condition does not
hold. On the other hand, even though this condition turned out to be quite restrictive,
Marquis provided in \cite{geohyperbolic} a large class of examples
(including groups having the Kazhdan property) which act geometrically on
hyperbolic graphs where the condition does hold, thereby establishing the hyperfiniteness of the corresponding boundary actions.
 
In this paper we solve the problem of hyperfiniteness of
boundary actions of hyperbolic groups in full generality, by proving the
following unconditional theorem.

\begin{thmintro}\label{thmintro:mainthm}
Let $G$ be a finitely generated hyperbolic group. Then the action of $G$ on its Gromov boundary $\partial G$ is hyperfinite.
\end{thmintro}

Note that for a finitely generated hyperbolic group $G$, any geometric action
of $G$ on a space $X$ induces a boundary action on
$\partial X$ and for all such actions there exists a
$G$-equivariant homeomorphism of $\partial X$ and $\partial G$
\cite{gromov}. Therefore, by Theorem~\ref{thmintro:mainthm}, all such
boundary actions of $G$ also induce hyperfinite
equivalence relations.

To prove Theorem~\ref{thmintro:mainthm}, we establish a new hyperfiniteness criterion of geometric nature and we show that it holds in every (uniformly) locally finite hyperbolic graph. More precisely, given such a graph $X$ with set of vertices $X^{(0)}$, we construct for each $x\in X^{(0)}$ and $\eta\in\partial X$ a subset $\Geo_1(x,\eta)$ of $\Geo(x,\eta)$ containing a sub-geodesic ray of every geodesic ray from $x$ to $\eta$, and such that $\Geo_1(x,\eta)$ and $\Geo_1(y,\eta)$ have a finite symmetric difference for every $x,y\in X^{(0)}$ (see Proposition~\ref{prop:Geo1_big} and Theorem~\ref{theorem:FSD}). We then show, in \S\ref{sec:endgame}, that if $X$ is the Cayley graph of a finitely generated hyperbolic group $G$, then this property implies the hyperfiniteness of the boundary action of $G$. 

As illustrated by Touikan's examples \cite{touikan} (see also Examples~\ref{example:bad_ladder} and \ref{example:bad_ladder2} below), establishing the finite symmetric difference property for the sets $\Geo_1(x,\eta)$ in \emph{arbitrary} (hyperbolic, uniformly locally finite) graphs $X$ is a rather subtle problem. The key idea to tackle this problem is to consider the \emph{horoboundary} of $X$, which is a refinement of $\partial X$ allowing for a better control of geodesic rays (and generalising the \emph{combinatorial compactification} of the chamber graph of a building introduced in \cite{CL11}), and to prove the existence of so-called \emph{straight} geodesic rays (see Definition~\ref{definition:straight}).

\subsection*{Acknowledgement}
We would like to thank the referee for many valuable comments.


\section{Preliminaries}\label{section:Preliminaries}

\subsection{Graphs}
Throughout this paper, $X$ will denote a connected locally finite graph, with vertex set $X^{(0)}$ and edge set $E(X)\subseteq X\times X$. 

Two distinct vertices $x,y\in X^{(0)}$ are {\bf adjacent} if $\{x,y\}\in E(X)$. A {\bf path} in $X$ is a (possibly infinite) ordered sequence $x_0,x_1,\dots$ of vertices of $X$ such that $x_i$ and $x_{i+1}$ are distinct and adjacent for each $i$. The path metric on $X^{(0)}$ will be denoted $\dc\co X^{(0)}\times X^{(0)}\to\NN$. Given two vertices $x,y\in X^{(0)}$, we let $$\Gamma(x,y):=\{z\in X^{(0)} \ | \ \dc(x,y)=\dc(x,z)+\dc(z,y)\}$$ be the union of all geodesic paths from $x$ to $y$. Given a (finite) path $\Gamma_1$ ending at some vertex $x\in X^{(0)}$, and a (possibly infinite) path $\Gamma_2$ starting at $x$, we denote by $\Gamma_1\cdot \Gamma_2$ the path obtained by concatenating $\Gamma_1$ with $\Gamma_2$.

A {\bf geodesic ray} is an infinite geodesic path $(x_n)_{n\in\NN}\subseteq X^{(0)}$. Two geodesic rays $r,r'$ are {\bf asymptotic} if they are at bounded Hausdorff distance $\dist_H(r,r')$. The {\bf visual boundary} $\partial X$ of $X$ is the set of equivalence classes of asymptotic geodesic rays. A {\bf CGR from $x\in X^{(0)}$ to $\eta\in\partial X$} is a geodesic ray $(x_n)_{n\in\NN}\subseteq X^{(0)}$ starting at $x$ and pointing towards $\eta$ (i.e. belonging to the equivalence class $\eta$); their set is denoted $\CGR(x,\eta)$. [The letters CGR stand for ``combinatorial geodesic ray'', emphasising the combinatorial nature of the metric.] We also let $\Geo(x,\eta)\subseteq X^{(0)}$ denote the union of all CGR from $x\in X^{(0)}$ to $\eta\in\partial X$. 

\subsection{Cayley graphs}
If $X$ is the Cayley graph of a finitely generated group $G$ with respect to a finite symmetric generating set $S$ (so that $X^{(0)}=G$), then given a path $\Gamma=(x_n)_{0\leq n\leq\ell}$ (resp. $\Gamma=(x_n)_{n\in\NN}$) in $X$, we define the {\bf type} $\mathrm{typ}(\Gamma)$ of $\Gamma$ as $$\mathrm{typ}(\Gamma):=(x_n^{-1}x_{n+1})_{0\leq n\leq \ell-1}\in S^{\ell}, \quad\textrm{resp.} \ \mathrm{typ}(\Gamma):=(x_n^{-1}x_{n+1})_{n\in\NN}\in S^{\NN}.$$

\subsection{Hyperbolicity}\label{subsection:Hyperbolicity}
The graph $X$ is called {\bf (Gromov) hyperbolic} if there is some $\delta>0$ such that every geodesic triangle $\Delta$ in $X$ is $\delta$-slim, that is, such that each side of $\Delta$ is contained in the $\delta$-neighbourhood of the other two sides. In that case, we also call $X$ {\bf $\delta$-hyperbolic}. A finitely generated group $G$ is called {\bf hyperbolic} if it has a hyperbolic Cayley graph (with respect to a finite generating set $S$); since hyperbolicity is a quasi-isometry invariant, this does not depend on the choice of $S$.

The key property of hyperbolic graphs that we will need is the following.

\begin{lemma}\label{lemma:hyperbolic_basic_prop}
Assume that $X$ is $\delta$-hyperbolic for some $\delta>0$. Let $x\in X^{(0)}$, $\eta\in\partial X$ and $\Gamma=(x_n)_{n\in\NN},\Gamma'=(x_n')_{n\in\NN}\in\CGR(x,\eta)$. Then $\dc(x_n,x_n')\leq 2\delta$ for all $n\in\NN$. In particular, $\dist_H(\Gamma,\Gamma')\leq 2\delta$.
\end{lemma}
\begin{proof}
This follows from \cite[Lemma~III.3.3]{BHCAT0}.
\end{proof}

If $X$ is hyperbolic, one can equip its visual boundary $\partial X$ with a compact (metrisable) topology, defined as follows: given a base point $x\in X$, a sequence $\eta_n\in\partial X$ converges to some $\eta\in\partial X$ if and only if there exist CGR $\Gamma_n\in\CGR(x,\eta_n)$ such that every subsequence of $(\Gamma_n)_{n\in\NN}$ subconverges (i.e. admits a subsequence that converges) to a CGR $\Gamma\in\CGR(x,\eta)$. The resulting topological space, which we again denote by $\partial X$, is called the  {\bf Gromov boundary} of $X$ (see \cite[Section~III.3]{BHCAT0}).

\subsection{Horoboundary}
Throughout this paper, we fix a base point $z_0\in X^{(0)}$. Set $$\FFF(X,z_0):=\{f\co X^{(0)}\to\RR \ | \ |f(x)-f(y)|\leq \dc(x,y) \ \forall x,y\in X^{(0)}, \quad f(z_0)=0\}.$$ We equip $\FFF(X,z_0)$ with the topology of pointwise convergence. To each $x\in X^{(0)}$, we attach the function
$$f_x=f_{x,z_0}\co X^{(0)}\to\RR:y\mapsto \dc(x,y)-\dc(x,z_0),$$
so that $f_{x,z_0}\in \FFF(X,z_0)$. The map 
$$\iota\co X^{(0)}\to \FFF(X,z_0):x\mapsto f_{x,z_0}$$
is then continuous and injective, and we identify $X^{(0)}$ with its image (see e.g. \cite[\S 3]{CL11}). The {\bf horofunction compactification} $\Chor(X)$ of $X$ is the closure of $X^{(0)}$ in $\FFF(X,z_0)$; it is independent of the base point $z_0$ (as $\FFF(X,z_0)$ can be canonically identified with the space of $1$-Lipschitz functions $X^{(0)}\to\RR$, modulo the constant functions). The {\bf horoboundary} of $X$ is $\Chb(X):=\Chor(X)\setminus X^{(0)}$.

\subsection{Descriptive set theory}\label{sec:borel-equiv-relat}
A \textbf{standard Borel space} is a set $Z$ equipped with a
$\sigma$-algebra which can be obtained as the $\sigma$-algebra
of Borel sets from some Polish topology on $Z$. Examples of
standard Borel spaces include the discrete finite sets,
the discrete infinite countable set or the Cantor set
$2^\N$. Given a
standard Borel space $Z$, by a \textbf{Borel} set in $Z$ we
mean any set which belongs to the $\sigma$-algebra. 
 If $Z,Y$ are standard Borel spaces, there is a canonical standard
Borel space structure on $Z\times Y$, as well as on
$Z^{<\N}:=\bigcup_{n\in\N}Z^n$ and $Z^{\NN}$.

We will use the simple observation that the sets of the form
$\{x\in Z \ | \  \phi(z)\}$ are Borel if 
$\phi(x)$ is a first-order formula where the predicates correspond to closed
or open subsets of Polish spaces and all the quantifiers
range over finite or countable sets. We provide the following
example for the benefit of readers not familiar with descriptive set theory.

\begin{example}
Let $G$ be a finitely generated group with a finite
symmetric generating set $S$. Fix a total order on $S$
and consider the induced lexicographical order $<$ on $S^n$. Let
$n\in\NN$ and $L_n\subseteq
G^\N\times G^n$ be the set of pairs $(\Gamma,s)$ such that
$\Gamma$ is a geodesic ray starting at $e$ (the neutral element of $G$) and
$s\in S^n$ is the lexicographically least string which
appears infinitely often as a substring of $\mathrm{typ}(\Gamma)\in S^{\NN}$. Then the set $L_n$ is Borel in $G^\N\times G^n$ as
witnessed by the following formula $\phi$: writing $\Gamma=(\Gamma_m)_{m\in\NN}$ and $s=(s_m)_{0\leq m< n}$, we have $(\Gamma,s)\in L_n$ if and only if 
\begin{align*}
\phi(\Gamma,s)\equiv\Gamma\mbox{ is a geodesic path from $e$ and }
\forall k\in\N \, \exists l>k\ \forall i<n\ 
\Gamma_{l+i}^{-1}\Gamma_{l+i+1}=s_i\\ \mbox{ and }\forall t\in S^n \mbox{ if
}t<s\mbox{ then }\exists k\in \N \, \forall l>k \ \exists i<n\
\Gamma_{l+i}^{-1}\Gamma_{l+i+1}\not=t_{i}.
\end{align*}

Note that the predicate stating that $\Gamma$ is a geodesic path from $e$
corresponds to a closed set $A$ in $G^\N$, the conditions
$\Gamma_{l+i}^{-1}\Gamma_{l+i+1}=s_i$ on $(\Gamma,s)$ describe open and closed sets $B_{l,i}$ in
$G^\N\times S^n$, and the conditions $\Gamma_{l+i}^{-1}\Gamma_{l+i+1}\not=t_{i}$ on $\Gamma$ describe open and closed sets $C_{l,i,t}$ in $G^{\NN}$. Hence, setting $D_t:=\{s\in S^n \ | \ s\leq t\}\subseteq G^n$ for each $t\in S^n$, and denoting by $\proj_{G^\NN}$ and $\proj_{G^n}$ the natural projections from $G^{\NN}\times G^n$ to $G^{\NN}$ and $G^n$, respectively, we have an explicit Borel description of $L_n$ as
$$L_n=\proj_{G^\NN}^{-1}(A)\cap \bigcap_{k\in\NN}\bigcup_{l>k}\bigcap_{i<n}B_{l,i} \cap \bigcap_{t\in S^n}\Big(\proj_{G^n}^{-1}(D_t)\cup \bigcup_{k\in\NN}\bigcap_{l>k}\bigcup_{i<n}\proj_{G^{\NN}}^{-1}(C_{l,i,t})\Big).$$
\end{example}

In the above example we refer to the formula $\phi$ as 
a {\bf Borel definition} of $L_n$. In Section~\ref{sec:endgame} we
will use more sophisticated computations similar to the one
in the example above.

An \textbf{analytic} subset of a standard Borel space $Z$ is an
image of a Borel set by a Borel function (or equivalently, a
projection to $Z$ of a Borel set $B\subseteq Z\times Y$ for some standard Borel space $Y$). In
other words, a subset $A$ of $Z$ is analytic if it can be
written as $\{z\in Z \ | \ \exists y\in Y\ (z,y)\in B\}$ for a
Borel set $B\subseteq Z\times Y$. So analytic subsets are
those sets which are definable by formulas which have at
most one existential quantifier whose range is an
uncountable standard Borel space. Analytic sets are closed
under countable unions and intersections but not under
complements. In other words, if $A_n\subseteq Z$ are
analytic for $n\in\N$, then
$\{z\in Z \ | \  \exists n\in\NN \ z\in A_n\}$ and
$\{z\in Z \ | \  \forall n\in\NN\ z\in A_n\}$ are also analytic. A
subset of a standard Borel space is \textbf{coanalytic} if
its complement is analytic. In other words, coanalytic sets
are those sets which are definable by formulas which have at
most one universal quantifer whose range is an uncountable
Borel space. A classical result of Souslin \cite[Theorem
14.11]{kechris} states that a set is Borel if and only if it
is both analytic and coanalytic.

Given a Borel set $B\subseteq Z\times Y$ with countable
vertical sections $B_z=\{y\in Y \ | \ (z,y)\in B\}$, the projection
$\mathrm{proj}_{Z}(B)=\{z\in Z \ | \ \exists y
\in Y\ (z,y)\in B\}$ is still Borel by the {\bf Lusin--Novikov
theorem} \cite[Theorem 18.10]{kechris}. The same theorem says that in that case $B$ contains the
graph of  a Borel function from $\mathrm{proj}_{Z}(B)$ to $Y$. This
implies that $B$ is the
union of countably many graphs of Borel functions from
$\mathrm{proj}_{Z}(B)$
to $Y$, and if the sections $B_z$ ($z\in Z$) are of size at most $r$, then
$B$ is the union of $r$ Borel functions. This is the version of the Lusin--Novikov theorem that we will use in Section~\ref{sec:endgame}.

A \textbf{Borel equivalence relation} is an equivalence
relation $E$ on a standard Borel space $Z$ such that $E$ is
Borel as a subset of $Z\times Z$. If $A\subseteq Z$, we will write $E|A$ for the restriction of $E$ to $A$.
An equivalence relation is
\textbf{countable} (resp. \textbf{finite}) if all of its equivalence classes are
countable (resp. finite). An equivalence relation is
\textbf{hyperfinite} if it can be written as an increasing
union of finite equivalence relations.

Given two Borel equivalence relations $E$ on $Z$ and $F$ on
$Y$, a Borel function $f: Z\to Y$ is a {\bf homomorphism} if $z_1\mathrel{E}z_2$ $\implies$
$f(z_1)\mathrel{F}f(z_2)$ for all $z_1,z_2\in Z$. The Borel map $f$ is a \textbf{Borel
  reduction} from $E$ to $F$ if $z_1\mathrel{E}z_2$ $\iff$
$f(z_1)\mathrel{F}f(z_2)$ for all $z_1,z_2\in Z$. We say that a Borel equivalence
relation $E$ is \textbf{Borel reducible} to a Borel
equivalence relation $F$ if there exists a Borel reduction
from $E$ to $F$.

A Borel equivalence relation $E$ is \textbf{smooth} if it is
Borel reducible to the identity relation on a standard Borel
space. Every finite Borel equivalence relation is
smooth. The simplest non-smooth countable Borel equivalence
relation is $E_0$ defined on $2^\N$ by $x\mathrel{E_0}y$ if
$\exists n \ \forall m>n\ x(m)=y(m)$. The relation $E_0$ is hyperfinite. A countable
 Borel equivalence relation is hyperfinite if and
 only if it is Borel reducible to $E_0$. 
 
A Borel equivalence relation is \textbf{hypersmooth} if it can be written as an increasing union of smooth Borel equivalence relations. For instance, the relation $E_1$, defined on $(2^\N)^\N$ by $x\mathrel{E_1}y$ if $\exists n \ \forall m>n\ x(m)=y(m)$, is hypersmooth. A Borel equivalence relation is hypersmooth if and only if it is Borel reducible to $E_1$.

In the realm of countable Borel equivalence relations, the classes of hyperfinite and hypersmooth equivalence relations coincide: if a countable equivalence relation is hypersmooth, then it is hyperfinite.

An
\textbf{analytic equivalence relation} $E$ on $Z$ is an
equivalence relation such that $E$ is analytic as a subset
of $Z\times Z$. 

We will use the following consequence of the Second Reflection
theorem \cite[Theorem 35.16]{kechris}

\begin{lemma}\cite[Lemma 4.1]{HSS17}\label{reflection}
  Let $Z$ be a standard Borel space, $A\subseteq Z$ be
  analytic and let $E$ be an analytic equivalence relation
  on $Z$ such that there is some $K >1$ such that every
  $E|A$-class has size less than $K$.  Then there is a Borel
  equivalence relation $F$ on $Z$ with $E|A\subseteq F$ such
  that every $F$-class has size less than $K$.  
\end{lemma}

For more details regarding notation and standard facts in
descriptive set theory we refer the reader to \cite{kechris} or \cite{Kan08}.


\section{Preliminary lemmas and basic definitions}
\begin{lemma}\label{lemma:preparation}
Let $x\in X^{(0)}$, and let $\Gamma=(x_n)_{n\in\NN}$ be a CGR. Then there exists some $k\in\NN$ such that $\Gamma_{xx_k}\cdot (x_n)_{n\geq k}$ is a CGR for any geodesic path $\Gamma_{xx_k}$ from $x$ to $x_k$.
\end{lemma}
\begin{proof}
Note that the function $g\co\NN\to\mathbb{Z}: n\mapsto \dc(x,x_n)-n$ is non-increasing, as 
$\dc(x,x_n)\geq \dc(x,x_{n+1})-1$, and bounded from below, as $g(n)=\dc(x,x_n)-\dc(x_0,x_n)\geq -\dc(x,x_0)$. Hence $g$ is eventually constant, that is, there is some $k\in\NN$ such that $\dc(x,x_n)=\dc(x,x_k)+n-k$ for all $n\geq k$, as desired.
\end{proof}

\begin{lemma}\label{lemma:existence}
If $\Gamma=(x_n)_{n\in\NN}\subseteq X^{(0)}$ is a CGR, then $(x_n)_{n\in\NN}$ converges in $\Chor(X)$. We denote its limit by $\xi_{\Gamma}\in\Chb(X)$ and we say that $\Gamma$ {\bf converges} to $\xi_{\Gamma}$.
\end{lemma}
\begin{proof}
Let $y\in X^{(0)}$, and let us show that $(f_{x_n}(y))_{n\in\NN}$ is eventually constant, as desired. By Lemma~\ref{lemma:preparation}, there is some $N\in\NN$ such that $\Gamma_{yx_N}\cdot (x_n)_{n\geq N}$ and $\Gamma_{z_0x_N}\cdot (x_n)_{n\geq N}$ are CGR for some geodesic path $\Gamma_{yx_N}$ from $y$ to $x_N$ and some geodesic path $\Gamma_{z_0x_N}$ from $z_0$ to $x_N$. In particular, if $n\geq N$, then 
\begin{align*}
f_{x_n}(y)&=\dc(x_n,y)-\dc(x_n,z_0)\\
&=\dc(y,x_N)+\dc(x_N,x_n)-\dc(z_0,x_N)-\dc(x_N,x_n)=f_{x_N}(y),
\end{align*}
yielding the claim. 
\end{proof}

\begin{definition}
Given $\eta\in\partial X$, we set
$$\Xi(\eta):=\{\xi\in\Chb(X) \ | \ \textrm{$\xi=\xi_{\Gamma}$ for some $\Gamma\in\CGR(z_0,\eta)$}\}.$$
Note that $\Xi(\eta)$ is independent of the choice of $z_0$ by Lemma~\ref{lemma:preparation}: if $z_0'\in X^{(0)}$ and $\xi\in\Chb(X)$, then $\xi=\xi_{\Gamma}$ for some $\Gamma\in\CGR(z_0,\eta)$ if and only if $\xi=\xi_{\Gamma'}$ for some $\Gamma'\in\CGR(z_0',\eta)$.
\end{definition}

\begin{definition}
For $x\in X^{(0)}$ and $\xi\in\Xi(\eta)$, define the {\bf combinatorial sector} $$Q(x,\xi):=\{y\in X^{(0)} \ | \ \textrm{$y$ is on a CGR from $x$ to $\eta$ that converges to $\xi$}\}.$$ Note that $Q(x,\xi)$ is nonempty by Lemma~\ref{lemma:preparation}.
\end{definition}

\begin{example}\label{example:A2tilde}
Consider the Coxeter complex of type $\widetilde{A}_2$ (see e.g. \cite[\S 3]{BrownAbr}), that is, the tiling of the Euclidean plane by congruent equilateral triangles (see Figure~\ref{figure:A2tilde}). Let $X$ be the associated chamber graph, namely, the graph with vertex set the barycenters of these triangles, and with an edge between two barycenters if the corresponding triangles share a common edge. 

In this situation, the horofunction compactification of $X$ coincides with the combinatorial compactification of $X$ introduced in \cite[\S 2]{CL11}, and $\Chb(X)$ can be thought of as a refinement of $\partial X$ (see \cite[\S 3]{CL11}, and also \cite[Example~3.1]{geohyperbolic}).

An example of a CGR $\Gamma\in\CGR(x,\eta)$ for some vertex $x\in X^{(0)}$ and some direction $\eta\in\partial X$ is depicted on Figure~\ref{figure:A2tilde}. The set $\Xi(\eta)$ can be viewed as the set of ``strips'' delimited by two adjacent lines in the direction of $\eta$, or else as the set of vertices of a simplicial line ``transversal'' to the direction $\eta$ (the dashed line on Figure~\ref{figure:A2tilde}) --- see also \cite[Appendix~A]{geohyperbolic}.

An example of combinatorial sector $Q(x,\xi)$ for some $\xi\in\Xi(\eta)$ is depicted as the coloured area in Figure~\ref{figure:A2tilde} (note that in this example, the notion of combinatorial sector coincides with that introduced in \cite[\S 2.3]{CL11} --- see \cite[Theorem~3.11]{geohyperbolic}).
\end{example}

\begin{figure}
\centering
\begin{minipage}[t]{.7\textwidth}
   \centering
  \includegraphics[trim = 0mm 10mm 17mm 15.4mm, clip, width=11cm]{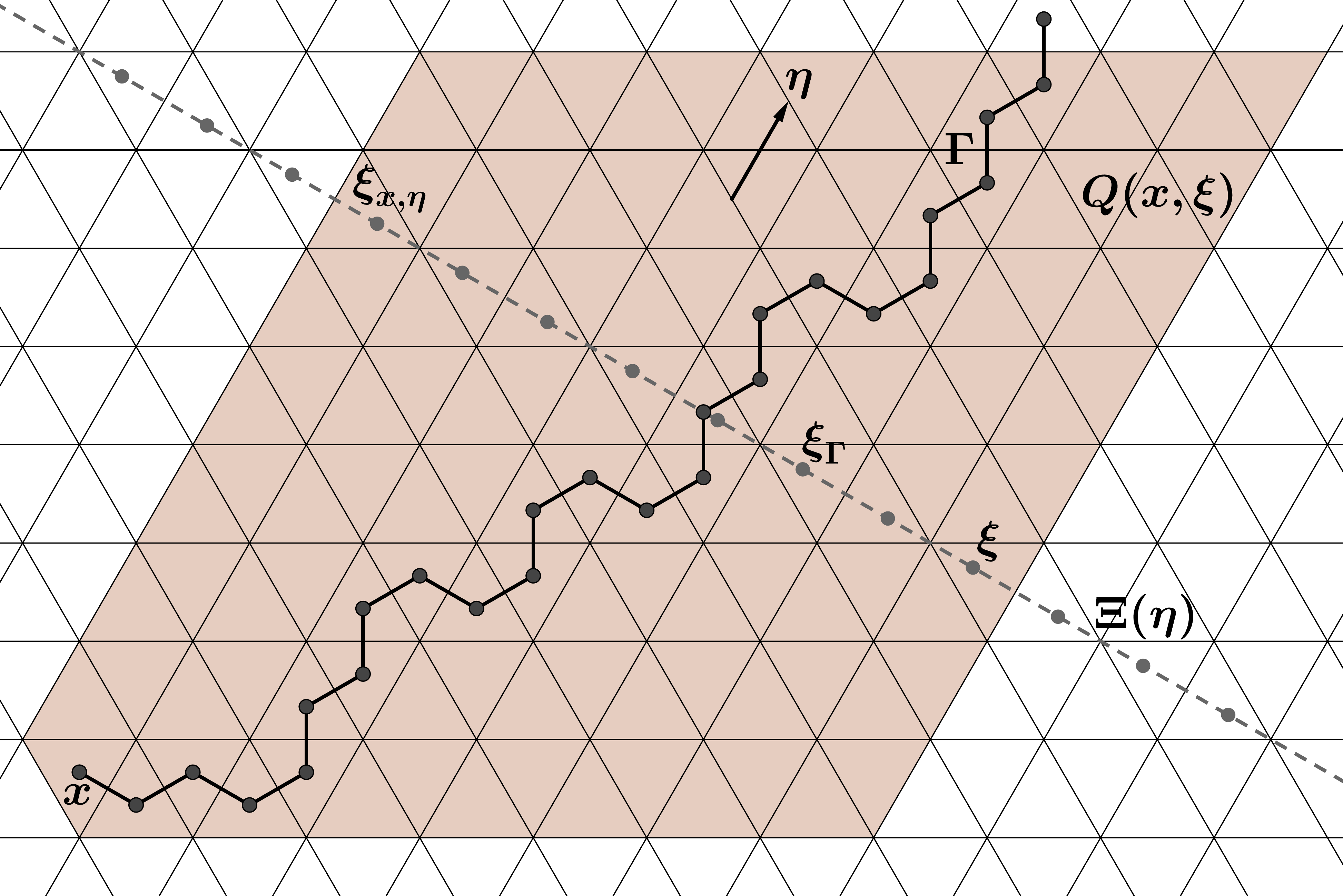}
  \captionof{figure}{Coxeter complex of type $\widetilde{A}_2$}
  \label{figure:A2tilde}
\end{minipage}%
\begin{minipage}[t]{.4\textwidth}
  \centering
  \includegraphics[trim = 54mm 95mm 103mm 45mm, clip, width=3cm]{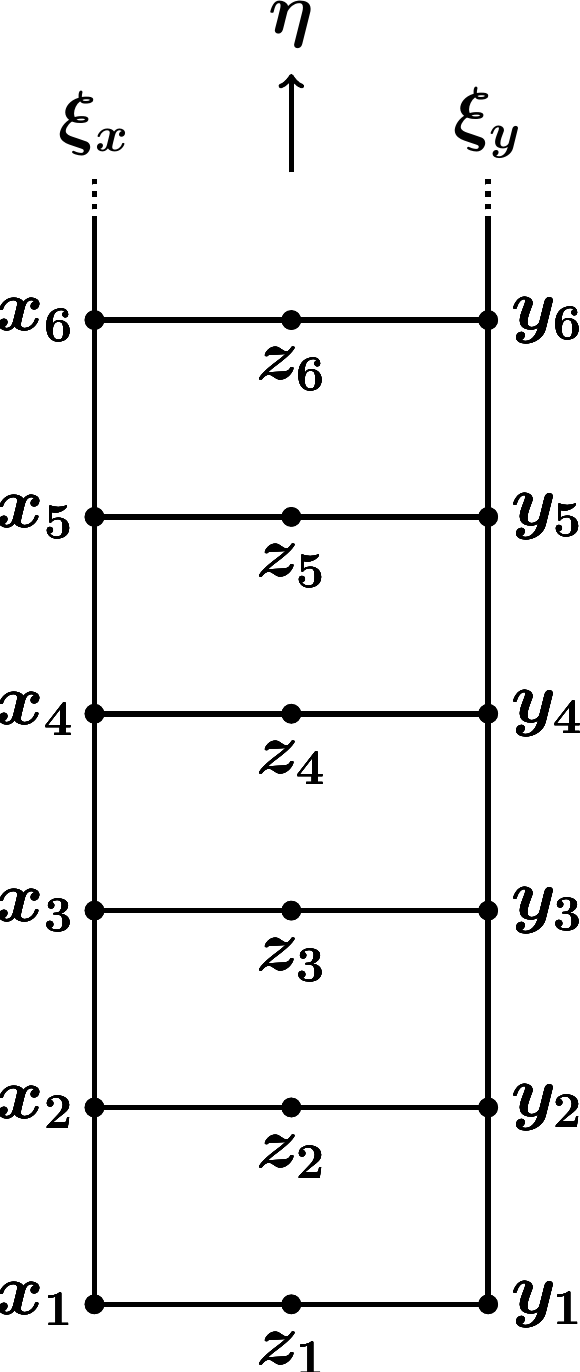}
  \captionof{figure}{Bad ladder}
  \label{figure:bad_ladder}
\end{minipage}
\end{figure}

\begin{example}\label{example:bad_ladder}
Consider the graph $X$ pictured in Figure~\ref{figure:bad_ladder}, with vertex set $X^{(0)}=\{x_n,y_n,z_n \ | \ n\geq 1\}$. Then $\partial X$ has a unique element $\eta\in\partial X$, whereas $\Chb(X)=\Xi(\eta)$ has two elements $\xi_x,\xi_y$, respectively corresponding to the (limits of the) CGR $(x_n)_{n\geq 1}$ and $(y_n)_{n\geq 1}$. Here are some examples of combinatorial sectors:
$$Q(x_1,\xi_x)=\{x_n \ | \ n\geq 1\}, \quad Q(x_1,\xi_y)=X^{(0)}, \quad\textrm{and}\quad Q(z_1,\xi_y)=\{z_1,y_n \ | \ n\geq 1\}.$$
Note also that $\Geo(x_1,\eta)=X^{(0)}$ and $\Geo(z_1,\eta)=\{z_1,x_n,y_n \ | \ n\geq 1\}$ have infinite symmetric difference, and hence $X$ does not satisfy the hyperfiniteness criterion given in \cite[Theorem~1.4]{HSS17}. This ``bad ladder'' is at the basis of the counter-example to \cite[Question~1.5]{HSS17} described in \cite{touikan}.
\end{example}


\section{Special vertices}
Throughout this section, we fix some $\eta\in\partial X$.

\begin{lemma}\label{lemma:convergence_same_xi}
Let $x\in X^{(0)}$ and $\xi\in\Xi(\eta)$. Let $(x_n)_{n\in\NN}\in\CGR(x,\eta)$ be converging to $\xi$ and let $y\in Q(x,\xi)$. Then there is some $N\in\NN$ such that $\dc(x,y)+\dc(y,x_n)=\dc(x,x_n)$ for all $n\geq N$.
\end{lemma}
\begin{proof}
Let $(y_n)_{n\in\NN}\in\CGR(x,\eta)$ be converging to $\xi$ and passing through $y$ (say $y=y_m$ for some $m\in\NN$). Since $(x_n)_{n\in\NN}$ and $(y_n)_{n\in\NN}$ both converge to $\xi$, there exists some $N\in\NN$ with $N\geq m$ such that 
$$f_{x_n}(y)=\xi(y)=f_{y_n}(y)\quad\textrm{and}\quad f_{x_n}(x)=\xi(x)=f_{y_n}(x)$$
for all $n\geq N$. Hence, for each $n\geq N$, we have
\begin{align*}
\dc(x,y)&=\dc(y_n,x)-\dc(y_n,y)=f_{y_n}(x)-f_{y_n}(y)=f_{x_n}(x)-f_{x_n}(y)\\
&=\dc(x_n,x)-\dc(x_n,y),
\end{align*}
yielding the claim.
\end{proof}

\begin{lemma}\label{lemma:alt_desc_Qxxi}
Let $x\in X^{(0)}$ and $\xi\in\Xi(\eta)$. Let $\Gamma=(x_n)_{n\in\NN}\in\CGR(x,\eta)$ be converging to $\xi$. Then
$$Q(x,\xi)=\bigcup_{n\in\NN}\Gamma(x,x_n).$$
\end{lemma}
\begin{proof}
The inclusion $\supseteq$ is clear. Conversely, let $y\in Q(x,\xi)$. Then Lemma~\ref{lemma:convergence_same_xi} yields some $N\in\NN$ such that $\dc(x,y)+\dc(y,x_N)=\dc(x,x_N)$, so that $y\in\Gamma(x,x_N)$.
\end{proof}

\begin{lemma}\label{lemma:Qx_in_Qy}
Let $x\in X^{(0)}$ and $\xi\in\Xi(\eta)$. If $y\in Q(x,\xi)$, then $Q(y,\xi)\subseteq Q(x,\xi)$.
\end{lemma}
\begin{proof}
Let $\Gamma_{xy}$ be a geodesic path from $x$ to $y$ and let $(y_n)_{n\in\NN}\in\CGR(y,\eta)$ be converging to $\xi$ and such that $\Gamma_{xy}\cdot (y_n)_{n\in\NN}$ is a CGR. Since $\Gamma(y,y_n)\subseteq\Gamma(x,y_n)$ for each $n\in\NN$, the lemma follows from Lemma~\ref{lemma:alt_desc_Qxxi}.
\end{proof}

\begin{lemma}\label{lemma:reorientate_CGR}
Let $x\in X^{(0)}$ and $\xi\in\Xi(\eta)$. Let $y\in Q(x,\xi)$, and let $\Gamma=(y_n)_{n\in\NN}\in\CGR(y,\eta)$ be such that $\Gamma\subseteq Q(y,\xi)$. Then the concatenation of any geodesic path from $x$ to $y$ with $\Gamma$ is a CGR.
\end{lemma}
\begin{proof}
Let $\Gamma'=(x_n)_{n\in\NN}$ be a CGR from $x$ to $\eta$ passing through $y$ and converging to $\xi$ (see Figure~\ref{figure:Lem3-4}). Let $r\in\NN$. We have to show that
$$\dc(x,y_r)=\dc(x,y)+\dc(y,y_r).$$
Since $y_r\in \Gamma\subseteq Q(y,\xi)$, we find by Lemma~\ref{lemma:convergence_same_xi} some $N_1\in\NN$ such that 
$$\dc(y,y_r)+\dc(y_r,x_n)=\dc(y,x_n)\quad\textrm{for all $n\geq N_1$.}$$
Similarly, since $y\in Q(x,\xi)$, we have $y_r\in Q(y,\xi)\subseteq Q(x,\xi)$ by Lemma~\ref{lemma:Qx_in_Qy}, and hence we find by  Lemma~\ref{lemma:convergence_same_xi} some $N_2\in\NN$ such that 
$$\dc(x,y_r)+\dc(y_r,x_n)=\dc(x,x_n)\quad\textrm{for all $n\geq N_2$.}$$
Let $N=\max\{N_1,N_2\}$. Then
\begin{align*}
\dc(x,y_r)-\dc(y,y_r)&=\dc(x,x_N)-\dc(y_r,x_N)-\dc(y,x_N)+\dc(y_r,x_N)\\
&=\dc(x,y),
\end{align*}
as desired.
\end{proof}

\begin{figure}
\centering
  \includegraphics[trim = 54mm 157mm 20mm 67mm, clip,width=10cm]{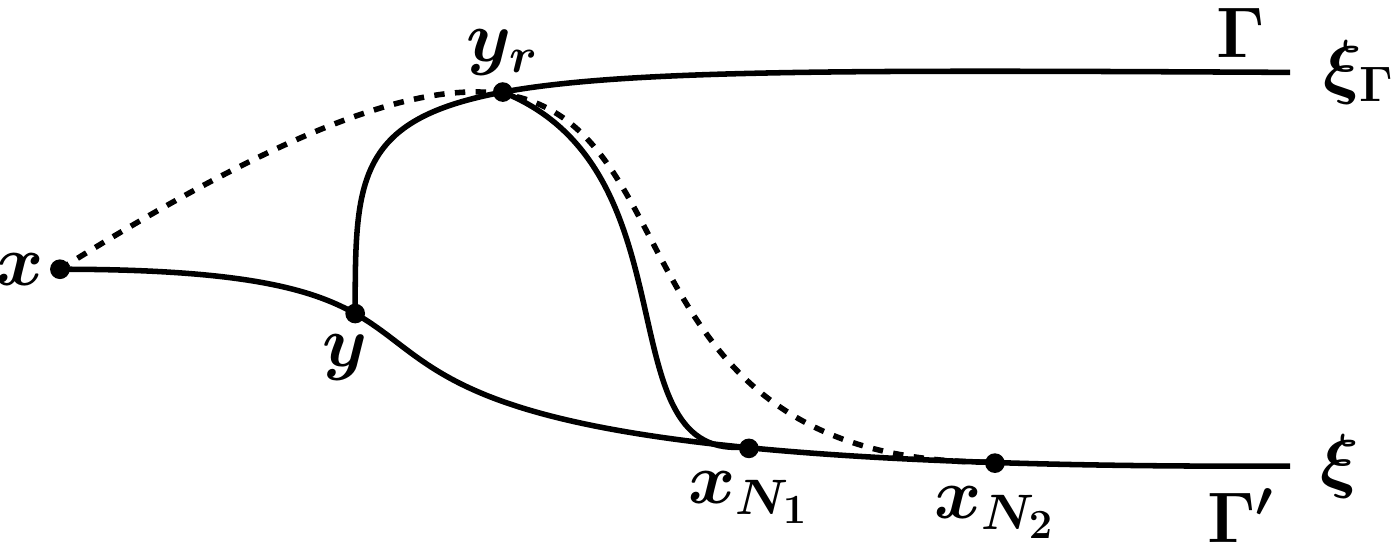}
  \captionof{figure}{Lemma~\ref{lemma:reorientate_CGR}}
  \label{figure:Lem3-4}
\end{figure}

\begin{lemma}\label{lemma:pre-min_stripe}
Let $x\in X^{(0)}$ and $\xi\in\Xi(\eta)$, and let $\Gamma\in\CGR(x,\eta)$. If $\Gamma\subseteq Q(x,\xi)$, then $Q(x,\xi_{\Gamma})\subseteq Q(x,\xi)$.
\end{lemma}
\begin{proof}
Write $\Gamma=(x_n)_{n\in\NN}$, and let $y\in Q(x,\xi_{\Gamma})$. By Lemma~\ref{lemma:convergence_same_xi}, we find some $N\in\NN$ such that 
$$\dc(x,y)+\dc(y,x_N)=\dc(x,x_N).$$
Let $(y_n)_{n\in\NN}\in \CGR(x,\eta)$ converge to $\xi$. Since $x_N\in \Gamma\subseteq Q(x,\xi)$, we find by Lemma~\ref{lemma:convergence_same_xi} some $M\in\NN$ such that 
$$\dc(x,x_N)+\dc(x_N,y_M)=\dc(x,y_M).$$
Hence
$$\dc(x,y)+\dc(y,x_N)+\dc(x_N,y_M)=\dc(x,y_M),$$
so that $y\in\Gamma(x,y_M)\subseteq Q(x,\xi)$, as desired. 
\end{proof}

\begin{lemma}\label{lemma:uniqueness_xi_Q}
Let $x\in X^{(0)}$. If $Q(x,\xi)=Q(x,\xi')$ for some $\xi,\xi'\in\Xi(\eta)$, then $\xi=\xi'$.
\end{lemma}
\begin{proof}
For any two vertices $a,b\in X^{(0)}$, we fix a geodesic path $\Gamma_{ab}$ from $a$ to $b$.
Let $(x_n)_{n\in\NN}\in\CGR(x,\eta)$ be converging to $\xi$, and $(y_n)_{n\in\NN}\in\CGR(x,\eta)$ be converging to $\xi'$ (see Figure~\ref{figure:Lem3-6}). We claim that for each $z\in X^{(0)}$, there exists some $N=N_z\in\NN$ such that
\begin{equation}\label{eqn:TPuniq}
\xi(z)=f_{x_n}(z), \quad \xi'(z)=f_{y_n}(z)\quad\textrm{and}\quad\dc(z,x_n)=\dc(z,y_n)\quad\textrm{for all $n\geq N$.}
\end{equation}
This will imply that for each $z\in X^{(0)}$,
$$\xi(z)=f_{x_n}(z)=\dc(x_n,z)-\dc(x_n,z_0)=\dc(y_n,z)-\dc(y_n,z_0)=f_{y_n}(z)=\xi'(z)$$
for all $n\geq \max\{N_z,N_{z_0}\}$, yielding the lemma.

Let thus $z\in X^{(0)}$. Let $N\in\NN$ be such that $\xi(z)=f_{x_n}(z)$ and $\xi'(z)=f_{y_n}(z)$ for all $n\geq N$, and such that 
\begin{equation}\label{eqn:feuille2}
\textrm{$\Gamma_{zx_N}\cdot (x_n)_{n\geq N}$ and $\Gamma_{zy_N}\cdot (y_n)_{n\geq N}$ are CGR}
\end{equation}
(see Lemma~\ref{lemma:preparation}). We will prove that $\dc(z,x_N)=\dc(z,y_N)$, yielding (\ref{eqn:TPuniq}).

Since $y_N\in Q(x,\xi')=Q(x,\xi)$, we find by Lemma~\ref{lemma:convergence_same_xi} some $N'>N$ such that
\begin{equation}\label{eqn:feuille3}
\dc(x,y_N)+\dc(y_N,x_{N'})=\dc(x,x_{N'}).
\end{equation}
In particular,
\begin{equation}\label{eqn:feuille3bis}
\dc(y_N,x_{N'})=\dc(x,x_{N'})-\dc(x,y_N)=\dc(x,x_{N'})-\dc(x,x_N)=\dc(x_N,x_{N'}).
\end{equation}
Similarly, since $x_{N'}\in Q(x,\xi)=Q(x,\xi')$, we find by Lemma~\ref{lemma:convergence_same_xi} some $N''>N'$ such that
\begin{equation}\label{eqn:feuille4}
\dc(x,x_{N'})+\dc(x_{N'},y_{N''})=\dc(x,y_{N''}).
\end{equation}
It then follows from (\ref{eqn:feuille3}) and (\ref{eqn:feuille4}) that
\begin{equation*}
\dc(y_N,x_{N'})+\dc(x_{N'},y_{N''})=-\dc(x,y_N)+\dc(x,y_{N''})=\dc(y_N,y_{N''}).
\end{equation*}
In particular, $x_{N'}$ is on the CGR $\Gamma_{y_N\xi'}:=\Gamma_{y_Nx_{N'}}\cdot \Gamma_{x_{N'}y_{N''}}\cdot (y_n)_{n\geq N''}$, so that $x_{N'}\in Q(y_N,\xi')$. As $y_N\in Q(z,\xi')$ by (\ref{eqn:feuille2}), it then follows from Lemma~\ref{lemma:reorientate_CGR} (applied to $x:=z$, $\xi:=\xi'$, $y:=y_N$ and $\Gamma:=\Gamma_{y_N\xi'}$) that 
\begin{equation}\label{eqn:feuille5}
\dc(z,y_N)+\dc(y_N,x_{N'})=\dc(z,x_{N'}).
\end{equation}
As
$$\dc(z,x_N)+\dc(x_N,x_{N'})=\dc(z,x_{N'})$$
by (\ref{eqn:feuille2}), we deduce from (\ref{eqn:feuille5}) that
\begin{align*}
\dc(z,y_N)&=\dc(z,x_{N'})-\dc(y_N,x_{N'})=\dc(z,x_N)+\dc(x_N,x_{N'})-\dc(y_N,x_{N'})\\
&=\dc(z,x_N),
\end{align*}
where the last equality follows from (\ref{eqn:feuille3bis}), as desired.
\end{proof}

\begin{figure}
  \centering
  \includegraphics[trim = 54mm 162mm 23mm 45mm, clip, width=10cm]{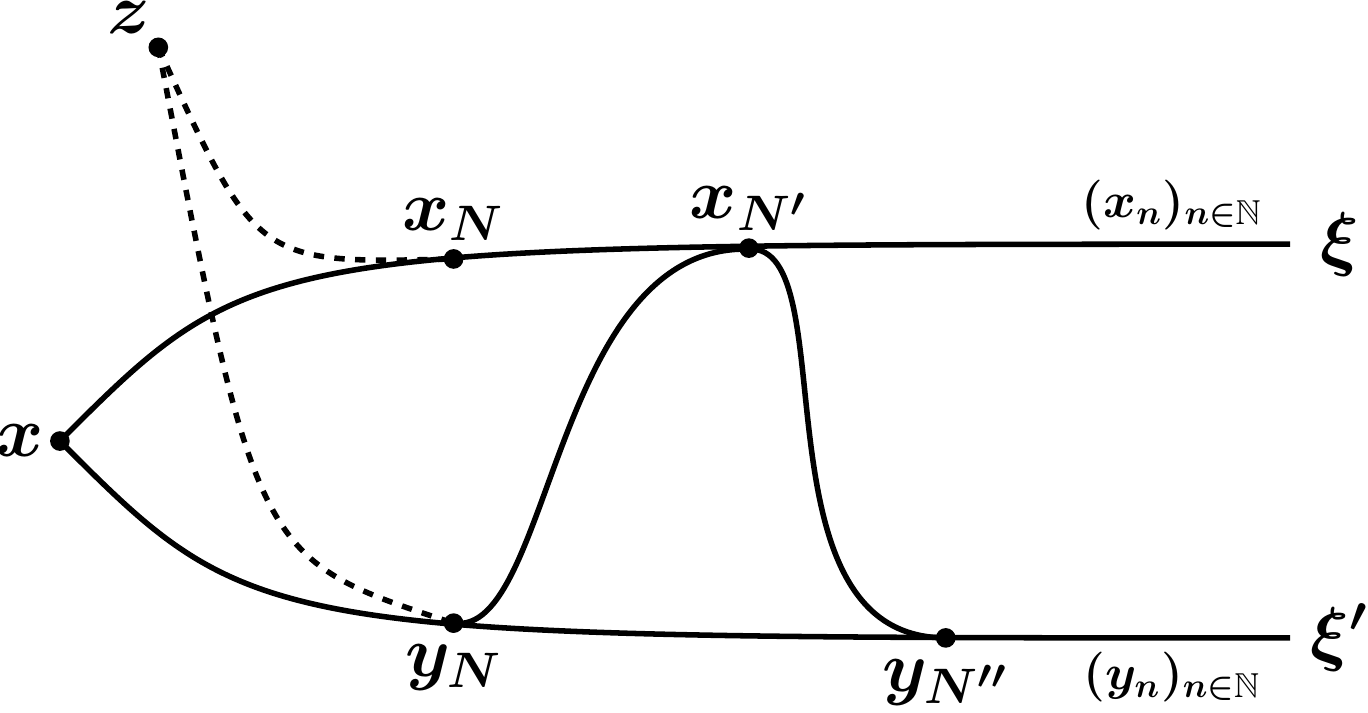}
  \captionof{figure}{Lemma~\ref{lemma:uniqueness_xi_Q}}
  \label{figure:Lem3-6}
\end{figure}

\begin{lemma}\label{lemma:xixeta}
Let $x\in X^{(0)}$. If $\bigcap_{\xi\in\Xi(\eta)}Q(x,\xi)$ contains a CGR from $x$ to $\eta$, then there exists a unique $\xi_{x,\eta}\in\Xi(\eta)$ such that $Q(x,\xi_{x,\eta})=\bigcap_{\xi\in\Xi(\eta)}Q(x,\xi)$.
\end{lemma}
\begin{proof}
Let $\Gamma\in\CGR(x,\eta)$ be contained in $\bigcap_{\xi\in\Xi(\eta)}Q(x,\xi)$. Then $Q(x,\xi_{\Gamma})\subseteq  \bigcap_{\xi\in\Xi(\eta)}Q(x,\xi)$ by Lemma~\ref{lemma:pre-min_stripe}, and hence $Q(x,\xi_{\Gamma})= \bigcap_{\xi\in\Xi(\eta)}Q(x,\xi)$. The uniqueness statement follows from Lemma~\ref{lemma:uniqueness_xi_Q}.
\end{proof}

\begin{definition}\label{definition:special}
Call $x\in X^{(0)}$ {\bf $\eta$-special} if $\bigcap_{\xi\in\Xi(\eta)}Q(x,\xi)$ contains a CGR from $x$ to $\eta$. We let $\Xs$ denote the set of $\eta$-special vertices. If $x\in\Xs$, we let $\xi_{x,\eta}$ be the unique element of $\Xi(\eta)$ such that $Q(x,\xi_{x,\eta})=\bigcap_{\xi\in\Xi(\eta)}Q(x,\xi)$ (see Lemma~\ref{lemma:xixeta}).
\end{definition}

\begin{definition}\label{definition:straight}
Let $x\in X^{(0)}$. We call $\Gamma\in\CGR(x,\eta)$ {\bf straight} if $\Gamma\subseteq \bigcup_{n\in\NN}\Gamma(x,y_n)$ for all $(y_n)_{n\in\NN}\in\CGR(x,\eta)$.
\end{definition}

\begin{example}
In the context of Example~\ref{example:A2tilde}, every vertex $x\in X^{(0)}$ is $\eta$-special (see \cite[Lemma~3.3]{geohyperbolic}). The element $\xi_{x,\eta}\in\Xi(\eta)$ is represented on Figure~\ref{figure:A2tilde}; the combinatorial sector $Q(x,\xi_{x,\eta})$ in this case coincides with the unique straight CGR from $x$ to $\eta$ (contained in the ``stripe'' between the two adjacent lines in the direction of $\eta$ that are on each side of $x$).
\end{example}

\begin{example}
In the context of Example~\ref{example:bad_ladder} (see Figure~\ref{figure:bad_ladder}), the vertices $x_n$ and $y_n$ ($n\geq 1$) are $\eta$-special: we have $\xi_{x_n,\eta}=\xi_x$ and $\xi_{y_n,\eta}=\xi_y$ for all $n\geq 1$. The vertices $z_n$ ($n\geq 1$), on the other hand, are not $\eta$-special, as $Q(z_n,\xi_x)\cap Q(z_n,\xi_y)=\{z_n\}$ for each $n\geq 1$.
\end{example}

\begin{lemma}\label{lemma:straight_equival}
Let $x\in X^{(0)}$ and $\Gamma=(x_n)_{n\in\NN}\in\CGR(x,\eta)$. Then the following assertions are equivalent:
\begin{enumerate}
\item
$\Gamma$ is straight.
\item
$\Gamma\subseteq \bigcap_{\xi\in\Xi(\eta)}Q(x,\xi)$.
\item
For any $(y_n)_{n\in\NN}\in\CGR(y_0,\eta)$ and any $n\in\NN$, there is some $M\in\NN$ such that $\dc(x,x_n)+\dc(x_n,y_m)=\dc(x,y_m)$ for all $m\geq M$. 
\item
Every sub-CGR of $\Gamma$ (i.e. CGR of the form $(x_n)_{n\geq m}$ for some $m\in\NN$) is straight.
\end{enumerate}
\end{lemma}
\begin{proof}
(1)$\implies$(2): Let $\xi\in\Xi(\eta)$, and let $(y_n)_{n\in\NN}\in\CGR(x,\eta)$ be converging to $\xi$. Then $\Gamma\subseteq \bigcup_{n\in\NN}\Gamma(x,y_n)$ by assumption and $\bigcup_{n\in\NN}\Gamma(x,y_n)=Q(x,\xi)$ by Lemma~\ref{lemma:alt_desc_Qxxi}, yielding (2).

(2)$\implies$(3): Let $\Gamma'=(y_n)_{n\in\NN}\in\CGR(y_0,\eta)$ and let $n\in\NN$. By assumption, $x_n\in\Gamma\subseteq Q(x,\xi_{\Gamma'})$. By Lemma~\ref{lemma:preparation}, there is some $M\in\NN$ such that the concatenation $\widetilde{\Gamma}:=\Gamma_{x_ny_M}\cdot (y_m)_{m\geq M}$ of a geodesic $\Gamma_{x_ny_M}$ from $x_n$ to $y_M$ with the CGR $(y_m)_{m\geq M}$ is again a CGR. Hence (3) follows from Lemma~\ref{lemma:reorientate_CGR} (with $x:=x$, $\xi:=\xi_{\Gamma'}$, $y:=x_n$ and $\Gamma:=\widetilde{\Gamma}$).

(3)$\implies$(1): Let $(y_n)_{n\in\NN}\in\CGR(x,\eta)$. By assumption, we find for any $n\in\NN$ some $M\in\NN$ such that $x_n\in \Gamma(x,y_M)$. Hence $\Gamma\subseteq\bigcup_{M\in\NN}\Gamma(x,y_M)$, yielding (1).

(4)$\iff$(1): The implication (4)$\implies$(1) is trivial, and the converse is clear in view of the equivalence (1)$\iff$(3) that we have just established.
\end{proof}

\begin{lemma}\label{lemma:equivalence_special}
Let $x\in X^{(0)}$. Then the following assertions are equivalent.
\begin{enumerate}
\item
$x$ is $\eta$-special.
\item
There exists a straight $\Gamma\in\CGR(x,\eta)$.
\end{enumerate}
If the above assertions hold, then $\xi_{x,\eta}=\xi_{\Gamma}$.
\end{lemma}
\begin{proof}
This readily follows from the equivalence (1)$\iff$(2) in Lemma~\ref{lemma:straight_equival}.
\end{proof}

\begin{lemma}\label{lemma:redicrect_xetaxi}
Let $x\in X^{(0)}$ and let $y\in\Geo(x,\eta)$ be $\eta$-special. Then $\Gamma_{xy}\cdot\Gamma_{y\xi_{y,\eta}}\in\CGR(x,\eta)$ for any geodesic path $\Gamma_{xy}$ from $x$ to $y$ and any CGR $\Gamma_{y\xi_{y,\eta}}$ from $y$ converging to $\xi_{y,\eta}$.
\end{lemma}
\begin{proof}
Since $y\in Q(x,\xi_{\Gamma})$ where $\Gamma$ is a CGR from $x$ to $\eta$ passing through $y$, and since $\Gamma_{y\xi_{y,\eta}}\subseteq Q(y,\xi_{y,\eta})\subseteq Q(y,\xi_{\Gamma})$, this follows from 
Lemma~\ref{lemma:reorientate_CGR}.
\end{proof}

\begin{lemma}\label{lemma:left_inclusion}
Let $x\in \Xs$. Then $Q(x,\xi_{x,\eta})\subseteq \Xs$ and $\xi_{y,\eta}=\xi_{x,\eta}$ for all $y\in Q(x,\xi_{x,\eta})$.
\end{lemma}
\begin{proof}
Let $y\in Q(x,\xi_{x,\eta})$, and let $\Gamma$ be a CGR from $x$ converging to $\xi_{x,\eta}$ and passing through $y$. Thus $\Gamma\subseteq Q(x,\xi_{x,\eta})$ and hence $\Gamma$ is straight by Lemma~\ref{lemma:straight_equival}(1)$\Leftrightarrow$(2). In particular, the sub-CGR of $\Gamma$ starting at $y$ is straight by Lemma~\ref{lemma:straight_equival}(1)$\Leftrightarrow$(4), so that $y\in\Xs$ and $\xi_{y,\eta}=\xi_{x,\eta}$ by Lemma~\ref{lemma:equivalence_special}.
\end{proof}

\begin{definition}
For $x\in X^{(0)}$ and $\xi\in\Xi(\eta)$, we define the function $$\dist_{x,\xi}\co X^{(0)}\to\NN: a\mapsto \dc(x,a)+\xi(a)-\xi(x).$$
\end{definition}

\begin{lemma}\label{lemma:existence_special_vertices}
Let $x\in X^{(0)}$ and $\xi\in\Xi(\eta)$. Let also $\Gamma=(x_n)_{n\in\NN}\in\CGR(x,\eta)$. 
\begin{enumerate}
\item
The sequence $(\dist_{x,\xi}(x_n))_{n\in\NN}$ is non-decreasing.
\item
$\dist_{x,\xi}(x_n)\leq  2\cdot \dist_{H}(\Gamma,\Gamma')$ for any $n\in\NN$ and for any $\Gamma'\in\CGR(x,\eta)$  converging to $\xi$. 
\item
There is some $N\in\NN$ such that $\dist_{x,\xi}(x_n)=\dist_{x,\xi}(x_N)$ for all $n\geq N$. 
\item
If $N$ is as in (3), then for any $\Gamma'=(y_n)_{n\in\NN}\in\CGR(y_0,\eta)$ converging to $\xi$, and for any $n\geq N$, there exists some $M\in\NN$ such that $\dc(x_N,x_n)+\dc(x_n,y_m)=\dc(x_N,y_m)$ for all $m\geq M$.
\end{enumerate}
\end{lemma}
\begin{proof}
(1) Let $\Gamma'=(y_n)_{n\in\NN}\in\CGR(x,\eta)$ be converging to $\xi$. For all large enough $M\in\NN$, we have
\begin{align*}
\dist_{x,\xi}(x_{n+1})-\dist_{x,\xi}(x_n)&=\dc(x,x_{n+1})+\dc(x_{n+1},y_M)-\dc(x,x_n)-\dc(x_n,y_M)\\
&=1+\dc(x_{n+1},y_M)-\dc(x_n,y_M)\geq 0.
\end{align*}

(2) Let $\Gamma'=(y_n)_{n\in\NN}\in\CGR(x,\eta)$ be converging to $\xi$. Let $n\in\NN$. Then for all large enough $M\in\NN$, we have $$\dist_{x,\xi}(x_n)=\dc(x,x_n)+\dc(x_n,y_M)-\dc(x,y_M)\leq 2\cdot\dist(x_n,\Gamma')\leq 2\cdot \dist_{H}(\Gamma,\Gamma'),$$ as desired.

(3) Let $\Gamma'\in\CGR(x,\eta)$ be converging to $\xi$. Since $\dist_H(\Gamma,\Gamma')<\infty$, it follows from (2) that the set $\{\dist_{x,\xi}(x_n) \ | \ n\in\NN\}$ is finite. In view of (1), this implies that $(\dist_{x,\xi}(x_n))_{n\in\NN}$ is eventually constant, as desired.

(4) Let $(y_n)_{n\in\NN}\in\CGR(y_0,\eta)$ be converging to $\xi$ and $n\geq N$. Let $M\in\NN$ be such that $\xi(x_n)-\xi(x)=\dc(y_m,x_n)-\dc(y_m,x)$ and $\xi(x_N)-\xi(x)=\dc(y_m,x_N)-\dc(y_m,x)$ for all $m\geq M$. Then for all $m\geq M$, we have
\begin{align*}
\dc(x,x_n)+\dc(y_m,x_n)&=\dist_{x,\xi}(x_n)+\dc(y_m,x)=\dist_{x,\xi}(x_N)+\dc(y_m,x)\\
&=\dc(x,x_N)+\dc(y_m,x_N),
\end{align*}
and hence
\begin{align*}
\dc(x_N,x_n)+\dc(x_n,y_m)&=\dc(x_N,x_n)+\dc(x,x_N)-\dc(x,x_n)+\dc(y_m,x_N)\\
&=\dc(x_N,y_m),
\end{align*}
yielding the claim.
\end{proof}

\begin{lemma}\label{lemma:substraight}
Let $x\in X^{(0)}$ and $(x_n)_{n\in\NN}\in\CGR(x,\eta)$. If $\Xi(\eta)$ is finite, then there is some $N\in\NN$ such that $(x_n)_{n\geq N}$ is straight.
\end{lemma}
\begin{proof}
Assume that $\Xi(\eta)$ is finite. For each $\xi\in\Xi(\eta)$, let $N_{\xi}\in\NN$ be such that $\dist_{x,\xi}(x_n)=\dist_{x,\xi}(x_{N_{\xi}})$ for all $n\geq N_{\xi}$ (see Lemma~\ref{lemma:existence_special_vertices}(3)). Set $N:=\max\{N_{\xi} \ | \ \xi\in\Xi(\eta)\}$. Then $\Gamma_N:=(x_n)_{n\geq N}$ is straight, as follows from Lemma~\ref{lemma:straight_equival}(1)$\Leftrightarrow$(3) and Lemma~\ref{lemma:existence_special_vertices}(4).
\end{proof}


\section{Consequences of hyperbolicity}
Throughout this section, we fix some $\eta\in\partial X$, and we assume that $X$ is hyperbolic.

\begin{lemma}\label{lemma:hyperbolic_subcv_Gamma}
Let $x\in X^{(0)}$, and let $(\Gamma_n)_{n\in\NN}$ be a sequence of CGR from $x$ to $\eta$. Then $(\Gamma_n)_{n\in\NN}$ subconverges to some $\Gamma\in\CGR(x,\eta)$.
\end{lemma}
\begin{proof}
Since $X$ is locally finite, $(\Gamma_n)_{n\in\NN}$ subconverges to some CGR $\Gamma$ from $x$. Moreover, $\Gamma$ is contained in a tubular neighbourhood of $\Gamma_0$ by Lemma~\ref{lemma:hyperbolic_basic_prop}, as desired.
\end{proof}

$X$ is called {\bf uniformly locally finite} if there exists a constant $R$ such that each $x\in X^{(0)}$ is contained in at most $R$ edges.

\begin{prop}\label{prop:locfin_hyperb_implies_finite}
Assume that $X$ is uniformly locally finite. Then $\Xi(\eta)$ is finite. Moreover, there is a constant $R\in\NN$ independent of $\eta$ such that $\Xi(\eta)$ has at most $R$ elements.
\end{prop}
\begin{proof}
Let $\delta>0$ be such that $X$ is $\delta$-hyperbolic, and let $R\in\NN$ be such that each ball of radius $2\delta$ in $X^{(0)}$ contains at most $R$ vertices. We claim that $\Xi(\eta)$ has at most $R$ elements. Indeed, assume for a contradiction that there is some subset $\Xi=\{\xi_0,\xi_1,\dots,\xi_{R}\}\subseteq\Xi(\eta)$ with $R+1$ elements. Let $x\in X^{(0)}$. For each $\xi\in\Xi$, let $\Gamma_{\xi}=(x^{\xi}_n)_{n\in\NN}\in\CGR(x,\eta)$ converge to $\xi$ (see Figure~\ref{figure:Prop4-3}). By Lemma~\ref{lemma:existence_special_vertices}(3), there is some $N\in\NN$ such that 
\begin{equation}\label{eqn:xixi'}
\dist_{x,\xi'}(x^{\xi}_n)=\dist_{x,\xi'}(x^{\xi}_{N})\quad\textrm{for all $\xi,\xi'\in\Xi$ and all $n\geq N$.}
\end{equation}
Set $M:=N+2\delta$. Then by Lemma~\ref{lemma:hyperbolic_basic_prop}, the ball $B(x^{\xi_0}_M,2\delta)\subseteq X^{(0)}$ centered at $x^{\xi_0}_M$ and of radius $2\delta$ contains the vertex $x^{\xi}_{M}$ of $\Gamma_{\xi}$  for each $\xi\in\Xi$. 
We claim that the vertices $x^{\xi}_{M}$ for $\xi\in\Xi$ are pairwise distinct, yielding the desired contradiction. Assume that $x^{\xi}_{M}=x^{\xi'}_{M}=:y$ for some $\xi,\xi'\in\Xi$, and let us show that $\xi=\xi'$. By Lemma~\ref{lemma:uniqueness_xi_Q}, it is sufficient to show that $Q(y,\xi)=Q(y,\xi')$. By Lemma~\ref{lemma:pre-min_stripe}, it is then sufficient to show that $(x^{\xi}_{n})_{n\geq M}\subseteq Q(y,\xi')$ (the converse inclusion $Q(y,\xi)\supseteq Q(y,\xi')$ being obtained by exchanging the roles of $\xi$ and $\xi'$). Let thus $n\geq M\geq N$. In view of (\ref{eqn:xixi'}), we can apply Lemma~\ref{lemma:existence_special_vertices}(4) (with $x:=x$, $\xi:=\xi'$, $\Gamma:=\Gamma_{\xi}$ and $\Gamma':=\Gamma_{\xi'}$) to conclude that there exists some $M'\in\NN$ such that $\dc(y,x^{\xi}_n)+\dc(x^{\xi}_n,x^{\xi'}_m)=\dc(y,x^{\xi'}_m)$ for all $m\geq M'$. Thus, $x^{\xi}_n\in Q(y,\xi')$, as desired.
\end{proof}

\begin{figure}
\centering
  \includegraphics[trim = 54mm 160mm 22mm 50mm, clip, width=10cm]{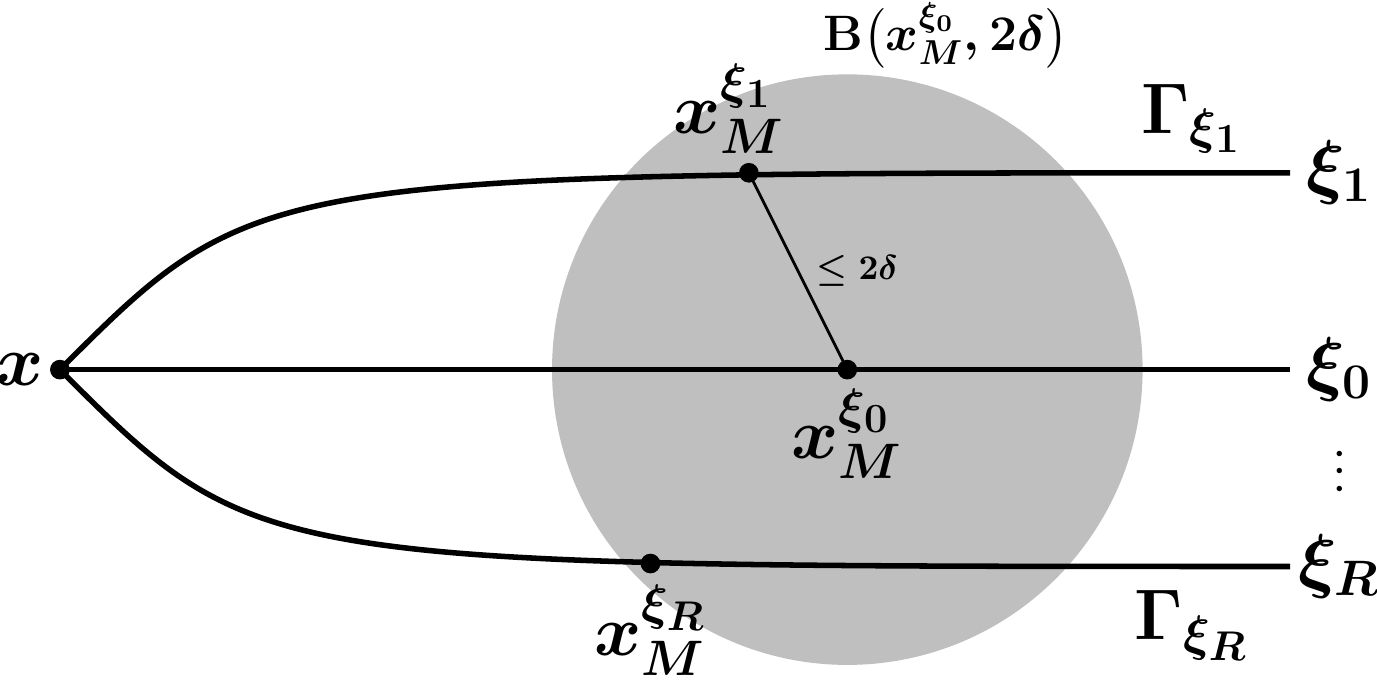}
  \captionof{figure}{Proposition~\ref{prop:locfin_hyperb_implies_finite}}
  \label{figure:Prop4-3}
\end{figure}

\begin{corollary}\label{corollary:special_exist}
Assume that $X$ is uniformly locally finite. Let $\Gamma=(x_n)_{n\in\NN}\in\CGR(x_0,\eta)$. Then there is some $N\in\NN$ such that $(x_n)_{n\geq N}$ is a straight CGR. In particular, $\Gamma\setminus\Xs$ is finite.
\end{corollary}
\begin{proof}
By Proposition~\ref{prop:locfin_hyperb_implies_finite}, the set $\Xi(\eta)$ is finite. Hence Lemma~\ref{lemma:substraight} yields some $N\in\NN$ such that $(x_n)_{n\geq N}$ is straight. The second claim then follows from Lemma~\ref{lemma:equivalence_special}. 
\end{proof}

\begin{lemma}\label{lemma:first_step_thm}
Let $x\in \Xs$ and $y\in X^{(0)}$. Then $Q(x,\xi_{x,\eta})\setminus  Q(y,\xi_{x,\eta})$ is finite.
\end{lemma}
\begin{proof}
Assume for a contradiction that there exists an (unbounded) sequence $(x_n)_{n\in\NN}\subseteq Q(x,\xi_{x,\eta})\setminus  Q(y,\xi_{x,\eta})$. In particular, $x_n\in\Xs$ and $\xi_{x_n,\eta}=\xi_{x,\eta}$ for all $n\in\NN$ by Lemma~\ref{lemma:left_inclusion}. For each $n\in\NN$, let $\Gamma_n\in\CGR(x,\eta)$ converge to $\xi_{x,\eta}$ and passing through $x_n$ (see Figure~\ref{figure:Lem4-5}). By Lemma~\ref{lemma:hyperbolic_subcv_Gamma}, we may assume, up to extracting a subsequence, that $(\Gamma_n)_{n\in\NN}$ converges to a CGR $\Gamma=(y_n)_{n\in\NN}$ from $x$ to $\eta$. Note that $\Gamma\subseteq Q(x,\xi_{x,\eta})$, so that $Q(x,\xi_{\Gamma})\subseteq Q(x,\xi_{x,\eta})$ by Lemma~\ref{lemma:pre-min_stripe}. Hence $Q(x,\xi_{\Gamma})= Q(x,\xi_{x,\eta})$ by definition of $\xi_{x,\eta}$, and $\xi_{\Gamma}=\xi_{x,\eta}$ by Lemma~\ref{lemma:uniqueness_xi_Q}. Let $N\in\NN$ be such that $y_N\in Q(y,\xi_{\Gamma})$ (see Lemma~\ref{lemma:preparation}), and let $M\in\NN$ be such that $y_N\in\Gamma_M\cap\Gamma(x,x_M)$. Thus $x_M\in Q(y_N,\xi_{x,\eta})$. On the other hand, Lemma~\ref{lemma:Qx_in_Qy} yields $Q(y_N,\xi_{\Gamma})\subseteq Q(y,\xi_{\Gamma})$. Therefore, $x_M\in Q(y,\xi_{x,\eta})$, yielding the desired contradiction.
\end{proof}

\begin{figure}
\centering
  \includegraphics[trim = 54mm 164mm 17mm 55mm, clip, width=10cm]{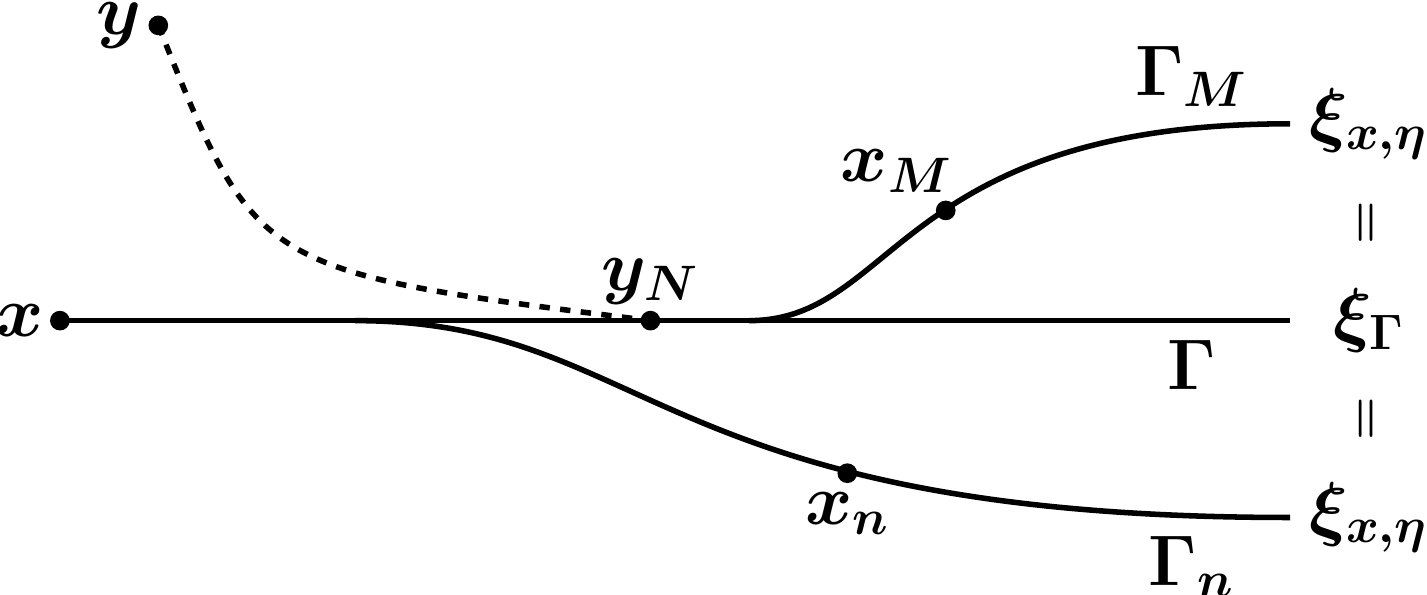}
  \captionof{figure}{Lemma~\ref{lemma:first_step_thm}}
  \label{figure:Lem4-5}
\end{figure}

\begin{definition}
Assume that $X$ is uniformly locally finite. Let $x\in X^{(0)}$. For $\xi\in\Xi(\eta)$, we let $Y(x,\xi)$ denote the set of $y\in \Geo(x,\eta)\cap\Xs$ with $\xi_{y,\eta}=\xi$ and such that $\dc(x,y)$ is minimal for these properties. Note that $Y(x,\xi)$ is finite (because $X$ is locally finite) and nonempty by Corollary~\ref{corollary:special_exist}. We set
$$\Geo_1(x,\eta):=\bigcup_{\xi\in\Xi(\eta)}\bigcup_{y\in Y(x,\xi)}Q(y,\xi).$$
\end{definition}

\begin{example}
In the context of Example~\ref{example:bad_ladder} (see Figure~\ref{figure:bad_ladder}), if $x\in\{x_m,y_m,z_m\}$  for some $m\geq 1$, then $\Geo_1(x,\eta)=\{x,x_n,y_n \ | \ n\geq m\}$.
\end{example}

\begin{example}\label{example:bad_ladder2}
Consider the graph $X$ depicted on Figure~\ref{figure:bad_ladder2}, with vertex set $X^{(0)}=\{x_n,y_n,z_n \ | \ n\geq 1\}$. As in Example~\ref{example:bad_ladder}, $\partial X$ has a unique element $\eta\in\partial X$, while $\Chb(X)=\Xi(\eta)$ has two elements $\xi_x,\xi_y$, respectively corresponding to the (limits of the) CGR $(x_n)_{n\geq 1}$ and $(y_n)_{n\geq 1}$. However, in this case, all vertices are $\eta$-special: if $m\geq 1$, then 
$(x_n)_{n\geq m}\in\CGR(x_m,\eta)$, $(y_n)_{n\geq m}\in\CGR(y_m,\eta)$ and $$ (z_m,y_{m+1})\cdot (y_n)_{n\geq m+1}\in\CGR(z_m,\eta)$$
are straight CGR. On the other hand, the sets $\Geo(x_1,\eta)=X^{(0)}\setminus\{y_1\}$ and $\Geo(y_1,\eta)=\{x_{n+1},y_n,z_{2n} \ | \ n\geq 1\}$ have an infinite symmetric difference. Note, however, that the sets $\Geo_1(x_1,\eta)=\{x_n,y_{n+1},z_1,z_{2n+2} \ | \ n\geq 1\}$ and $\Geo_1(y_1,\eta)=\{x_{n+1},y_n,z_{2n} \ | \ n\geq 1\}$ have finite symmetric difference.
\end{example}

\begin{prop}\label{prop:Geo1_big}
Assume that $X$ is uniformly locally finite. Let $x\in X^{(0)}$. Then $\Geo_1(x,\eta)\subseteq\Geo(x,\eta)\cap\Xs$, and for any $\Gamma\in\CGR(x,\eta)$, the set $\Gamma\setminus \Geo_1(x,\eta)$ is finite.
\end{prop}
\begin{proof}
The inclusion $\Geo_1(x,\eta)\subseteq\Geo(x,\eta)$ readily follows from Lemma~\ref{lemma:redicrect_xetaxi}, and the inclusion $\Geo_1(x,\eta)\subseteq\Xs$ from Lemma~\ref{lemma:left_inclusion}. Let now $\Gamma=(x_n)_{n\in\NN}\in\CGR(x,\eta)$. By Corollary~\ref{corollary:special_exist}, there is some $N\in\NN$ such that $(x_n)_{n\geq N}$ is a straight CGR. Note that $x_N\in \Xs$ and that $\xi:=\xi_{\Gamma}=\xi_{x_N,\eta}$ by Lemma~\ref{lemma:equivalence_special}. Moreover, $(x_n)_{n\geq N}\subseteq Q(x_N,\xi)$. Let $y\in Y(x,\xi)$. Then $Q(x_N,\xi)\setminus  Q(y,\xi)$ is finite by Lemma~\ref{lemma:first_step_thm}, and hence $(x_n)_{n\geq N}\setminus \Geo_1(x,\eta)$ is finite, yielding the claim.
\end{proof}

\begin{figure}
  \centering
  \includegraphics[trim = 15mm 181mm 21mm 45mm, clip, width=10cm]{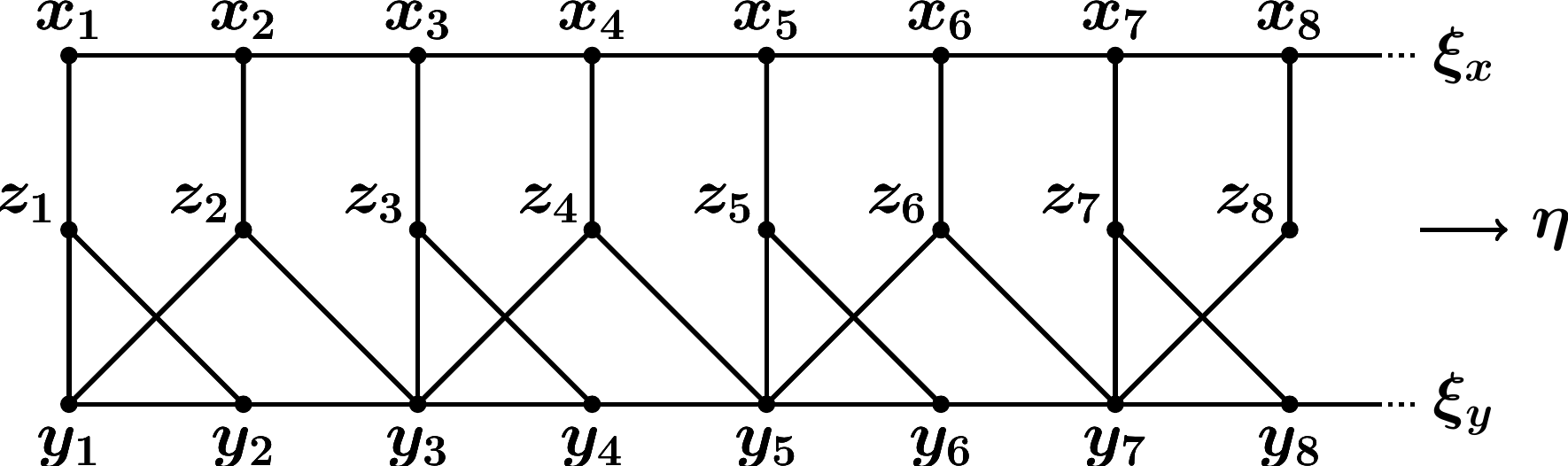}
  \captionof{figure}{Another ``bad ladder''}
  \label{figure:bad_ladder2}
\end{figure}

\begin{theorem}\label{theorem:FSD}
Assume that $X$ is hyperbolic and uniformly locally finite. Let $x,y\in X^{(0)}$. Then $\Geo_1(x,\eta)\setminus\Geo_1(y,\eta)$ is finite.
\end{theorem}
\begin{proof}
Assume for a contradiction that there exists an infinite sequence $(x_n)_{n\in\NN}\subseteq \Geo_1(x,\eta)\setminus \Geo_1(y,\eta)$. Since $\Xi(\eta)$ is finite by Proposition~\ref{prop:locfin_hyperb_implies_finite}, and since $Y(x,\xi)$ is finite for each $\xi\in\Xi(\eta)$ (because $X$ is locally finite), we may assume, up to extracting a subsequence, that $(x_n)_{n\in\NN}\subseteq Q(x',\xi)$ for some $\xi\in\Xi(\eta)$ and some $x'\in\Xs$ with $\xi_{x',\eta}=\xi$. Let $y'\in Y(y,\xi)$, so that $Q(y',\xi)\subseteq\Geo_1(y,\eta)$. Then Lemma~\ref{lemma:first_step_thm} implies that $Q(x',\xi)\setminus  Q(y',\xi)$ is finite, and hence $(x_n)_{n\in\NN}\setminus \Geo_1(y,\eta)$ is finite, yielding the desired contradiction.
\end{proof}

We conclude this section with an easy observation that will be used in \S\ref{sec:endgame}.

\begin{lemma}\label{lemma:Geo1_G_equiv}
Assume that $X$ is uniformly locally finite. Let $G$ be a subgroup of $\mathrm{Aut}(X)$. Then $g\cdot \Geo_1(x,\eta)=\Geo_1(gx,g\eta)$ for all $g\in G$ and $x\in X^{(0)}$.
\end{lemma}
\begin{proof}
Let $g\in G$. Since $g$ preserves geodesic paths and CGR, it acts on $\partial X$ and $\Chb(X)$, and we have $g\cdot\Xi(\eta)=\Xi(g\eta)$, $g\cdot Q(x,\xi)=Q(gx,g\xi)$ ($x\in X^{(0)}$, $\xi\in\Xi(\eta)$), $g\cdot \Xs=\mathrm{X^{(0)}_{s,g\eta}}$, $g\cdot \xi_{x,\eta}=\xi_{gx,g\eta}$ ($x\in\Xs$), $g\cdot\Geo(x,\eta)=\Geo(gx,g\eta)$ ($x\in X^{(0)}$), $g\cdot Y(x,\xi)=Y(gx,g\xi)$ ($x\in X^{(0)}$, $\xi\in\Xi(\eta)$), and hence $g\cdot \Geo_1(x,\eta)=\Geo_1(gx,g\eta)$ ($x\in X^{(0)}$).
\end{proof}


\section{Endgame}\label{sec:endgame}
This section is devoted to the proof of Theorem~\ref{thmintro:mainthm}.

Let $G$ be a finitely generated hyperbolic group. Write $e$ for the neutral element of $G$. Let $S$ be a finite symmetric set of generators of $G$, and let $X$ be the Cayley graph of $G$ with respect to $S$. Thus $X$ is a uniformly locally finite $\delta$-hyperbolic graph (for some $\delta>0$) with $X^{(0)}=G$, and we keep the notions and notations relative to $X$ that we developed in the previous sections, choosing $z_0:=e$ as base point (we will also write $\partial G$ and $\Chb(G)$ instead of $\partial X$ and $\Chb(X)$). In addition, given a path $\Gamma=(x_m)_{0\leq m\leq\ell}$ (resp. $\Gamma=(x_m)_{m\in\NN}$) in $X$, we denote the $n$-th entry $x_n$ of $\Gamma$ by $\Gamma_n$. Fix a total order $\leq$ on $G$ (and hence, in particular, on $S$) such that $$v\leq w \implies \dc(e,v)\le
\dc(e,w)\quad\textrm{for all $v,w\in G$.}$$

\begin{definition}\label{definition:Ceta}
Given a boundary point $\eta\in\partial G$,
we define $C^\eta\subseteq G\times S^{<\N}$ as
\begin{align*}
C^\eta := \{(\Gamma_0,\mathrm{typ}(\Gamma_0,\dots,\Gamma_{n}))\in G\times S^{<\N} \ | \ \textrm{$\Gamma\in\CGR(\Gamma_0,\eta)$, $\Gamma_0\in \mathrm{Geo}_1(e,\eta)$, $n\in\NN$}\}.
\end{align*}

Write $s^\eta_n\in S^{<\N}$ for the lexicographically least string of length $n$
which appears infinitely many times in $C^\eta$,
i.e. such that $(g,s^\eta_n)\in C^\eta$ for infinitely many $g\in \mathrm{Geo}_1(e,\eta)$
 (note that $C^{\eta}$ contains infinitely many elements of $G\times S^n$ for every $n\in\NN$ by Proposition~\ref{prop:Geo1_big}).
Note also that for each $n\in\NN$, the sequence $s^\eta_n$ is an initial
segment of the sequence $s^\eta_{n+1}$ (because if $(\Gamma_0,\mathrm{typ}(\Gamma_0,\dots,\Gamma_{n}))\in C^{\eta}$, then also $(\Gamma_0,\mathrm{typ}(\Gamma_0,\dots,\Gamma_{n+1}))\in C^{\eta}$). In particular, $s^\eta:=\lim_n s^\eta_n\in S^\N$ is well-defined.
Write $$T^\eta_n := \{g\in \mathrm{Geo}_1(e,\eta) \ | \ (g,s^\eta_n)\in C^\eta\}$$
and let $g^\eta_n := \min T^\eta_n$ (with respect to the total order on $G$). Finally, set $$k^\eta_n := \dc(e,g^\eta_n)$$ and note that the sequence $k^\eta_n$ is nondecreasing in $n$.
\end{definition}

We will now establish the Borelness of a number of subsets of standard Borel spaces constructed from the standard Borel spaces $G$, $\partial G$ and $\Chb(G)$. We recall that $G$ has the discrete topology, $\partial G$ the topology defined in \S\ref{subsection:Hyperbolicity}, and $\Chb(G)$ the topology of pointwise convergence.

\begin{claim}\label{claim:C}
The set $$C:=\{\Gamma\in G^\N \ | \  \textrm{$\Gamma$ is a CGR}\}$$ is closed in $G^\N$, and each set $C_g:=\{\Gamma\in C \ | \ \Gamma_0=g\}$ ($g\in G$) is compact.
\end{claim}
\begin{proof}
Each $C_g$ is compact by Lemma~\ref{lemma:hyperbolic_subcv_Gamma}, and hence $C$ is closed, as its intersection with each clopen subset $\{\Gamma\in G^\N \ | \ \Gamma_0=g\}$ ($g\in G$) of $G^\N$ is compact (hence closed). 
\end{proof}

\begin{claim}\label{claim:R}
The set $$R:=\{(\eta,g,\Gamma)\in\partial G\times G^{\mathbb{N}} \ | \ \Gamma\in\CGR(g,\eta)\}$$ is closed in $\partial G\times G^{\mathbb{N}}$.
\end{claim}
\begin{proof}
Indeed, if $\Gamma_n\in\CGR(g,\eta_n)$ converges to $\Gamma\in\CGR(g,\eta)$ and $\eta_n$ converges to $\eta'\in\partial G$, then by definition of the topology on $\partial X$, there is a sequence of CGR $\Gamma'_n\in\CGR(g,\eta_n)$ subconverging to a CGR $\Gamma'\in\CGR(g,\eta')$. As $\dist_H(\Gamma_n,\Gamma'_n)\leq 2\delta$ by Lemma~\ref{lemma:hyperbolic_basic_prop}, this implies that $\dist_H(\Gamma,\Gamma')\leq 2\delta$, and hence that $\eta=\eta'$, as desired.
 \end{proof}

\begin{claim}\label{setf}
 The set $$F:=\{(\eta,g,(\Gamma_0,\ldots,\Gamma_n))\in\partial G\times G^{<\NN} \ | \ \textrm{$\Gamma\in\CGR(g,\eta)$}\}$$ is Borel in $\partial G\times  G^{<\mathbb{N}}$.
\end{claim}
\begin{proof}
Since $F=\bigcup_{g\in G}F_g$, where $F_g:=\{(\eta,h,(\Gamma_0,\ldots,\Gamma_n))\in F \ | \ h=g\}$, it is sufficient to check that $F_g$ is closed for each $g\in G$. Let thus $g\in G$.
  Note that $F_g$ is the projection of the subset
  $$F_g'=\{(\eta,g,(\Gamma_0,\ldots,\Gamma_n),\Gamma')\in \partial G\times
  G^{<\mathbb{N}}\times C_g \ | \
  (\eta,g,\Gamma')\in R\mbox{ and } \forall i\leq n \ 
  \Gamma'_i=\Gamma_i\}$$ of $\partial G\times
  G^{<\mathbb{N}}\times C_g$ to $\partial G\times
  G^{<\mathbb{N}}$. As $F'_g$ is closed by Claim~\ref{claim:R}, and as the projection of a closed set along
  a compact space is also closed, $F_g$ is indeed closed by Claim~\ref{claim:C}.
\end{proof}

\begin{claim}\label{e}
  The set $$M:=\{(\eta,\xi)\in \partial G\times \mathcal{C}_{\mathrm{hb}}(G) \ | \ \textrm{$\xi\in\Xi(\eta)$}\}$$ is Borel in $\partial G\times \mathcal{C}_{\mathrm{hb}}(G)$.
\end{claim}
\begin{proof}
  We will show that $M$ is both analytic and coanalytic, hence Borel.

  Note first that if $\eta\in\partial G$ and
  $\xi\in\Chb(G)$, then $(\eta,\xi)\in M$ if and only if
  $$\exists \Gamma\in G^{\mathbb{N}} \ (\eta,\Gamma_0,\Gamma)\in R \ \textrm{and} \ \xi_\Gamma=\xi.$$ As $\xi_{\Gamma}=\xi$ if and only if $\forall g\in G$ $\exists n\in\NN$ $\forall m\geq n$ $f_{\Gamma_m}(g)=\xi(g)$, which is a Borel definition of the subset $\{(\xi,\Gamma)\in \Chb(G)\times C \ | \ \xi_{\Gamma}=\xi\}$ of $\Chb(G)\times G^{\NN}$, we deduce from Claim~\ref{claim:R} that $M$ is analytic (because the above formula defining $M$ has only one existential quantifier, $\exists \Gamma\in G^{\mathbb{N}}$, ranging over an uncountable standard Borel space).
  
  Now, to see that $M$ is coanalytic,
  note that $(\eta,\xi)\in M$ if and only if for every $\Lambda\in\CGR(e,\eta)$, we can find a sequence of geodesic paths $\Gamma^k\in G^k$ starting at $e$ and contained in the $2\delta$-neighbourhood $N_{2\delta}(\Lambda)$ of $\Lambda$ (such a sequence subconverges to some $\Gamma\in\CGR(e,\eta)$ by local finiteness of $X$), and such that $f_{\Gamma^j_i}(g)=\xi(g)$ for all vertices $g\in G$ and all large enough $i,j$ (i.e. so that $\xi_{\Gamma}=\xi$). (The forward implication follows from Lemma~\ref{lemma:hyperbolic_basic_prop}.) In other words, $(\eta,\xi)\in M$ if and only if \begin{align*}
  & \forall \Lambda\in G^\N \ \textrm{if $(\eta,e,\Lambda)\in R$, then
  $\forall k\in\NN$ $\exists \Gamma^k\in G^k$ a geodesic path with $\Gamma^k_0=e$}\\
  &\textrm{such that
  $\Gamma^k\subseteq N_{2\delta}(\Lambda)$ and such that $\forall g\in G \ \exists
  n_g\in\N \ \forall i,j>n_g\,
  f_{\Gamma^j_i}(g)=\xi(g)$.}
  \end{align*}
   Note that the three subformulas ``$(\eta,e,\Lambda)\in R$'', ``$\Gamma^k\subseteq N_{2\delta}(\Lambda)$'' and ``$\forall i,j>n_g\,  f_{\Gamma^j_i}(g)=\xi(g)$'' in the above formula correspond to closed or closed and open sets (in appropriate standard Borel spaces) and all the quantifiers except for one range over
countable sets. This single quantifier is the universal one $\forall\Lambda\in G^\N$, and thus the above formula defines a coanalytic set.
\end{proof}

\begin{definition}
By Proposition \ref{prop:locfin_hyperb_implies_finite}, each $\Xi(\eta)$ is finite and there exists $r>0$ such
each $\Xi(\eta)$ has at most $r$ elements. Since $M$ is Borel in $\partial G\times \Chb(G)$ by Claim~\ref{e} and has
finite sections $M_{\eta}\subseteq \Chb(G)$ of size at most $r$, the Lusin--Novikov
theorem yields Borel functions $\xi_1,\dots,\xi_r\co \partial G\to \Chb(G)$ such that $M$ is the union of
their graphs $G_{\xi_i}:=\{(\eta,\xi_i(\eta))\in\partial G\times \Chb(G)\}$.
\end{definition}

\begin{claim}\label{ku}
For each $i\in\{1,\dots,r\}$, the set $$Q_i:=\{(\eta,g,h)\in\partial G\times G^2 \ | \ h\in Q(g,\xi_i(\eta))\}$$ is Borel in $\partial G\times G^2$.
\end{claim}
\begin{proof}
  By Lemma \ref{lemma:alt_desc_Qxxi},
  $Q(g,\xi_i(\eta))=\bigcup_{n\in\N}\Gamma(g,x_n)$ for some, or
  equivalently every CGR $(x_n)_{n\in\N}$ converging to
  $\xi_i(\eta)$. This easily gives both analytic and coanalytic
  definitions of $Q_i$, because $h\in Q(g,\xi_i(\eta))$ if and only if $\exists \Lambda\in C$ (resp. $\forall \Lambda\in C$) with $\Lambda_0=g$ and $\xi_{\Lambda}=\xi_i(\eta)$ and $\exists n\in\NN$ $h\in\Gamma(g,\Lambda_n)$.
 \end{proof}

\begin{claim}\label{claim:P}
The set $$P:=\{(\eta,h)\in\partial G\times G \ | \ \textrm{$h$ is $\eta$-special}\}$$ is Borel in $\partial G\times G$.
\end{claim}
\begin{proof}
  This follows from the definition of special vertices (see Definition~\ref{definition:special}) and
  the fact that the sets $F$ and $Q_i$ are Borel (see Claims~\ref{setf} and \ref{ku}). Indeed, $h$ is $\eta$-special
  if and only if 
  $$\forall n\in\NN \ \exists \Gamma^n\in G^n \ (\eta,h,\Gamma^n)\in F \ \textrm{and} \ \forall i\leq r \ \forall k<n\ \Gamma^n_k\in Q(h,\xi_i(\eta)).$$ The ``if'' in the above equivalence
  follows by local finiteness of $X$, since the sequence $(\Gamma^n)_{n\in\NN}$ with $(\eta,h,\Gamma^n)\in F$ will subconverge to some $\Gamma\in\CGR(h,\eta)$.
\end{proof}

\begin{claim}\label{pjeden}
  The set $$P_1:=\{(\xi,\eta,h)\in\Chb(G)\times\partial G\times G \ | \  \textrm{$h$ is $\eta$-special and $\xi=\xi_{h,\eta}$}\}$$ is Borel in $\Chb(G)\times\partial G\times G$.
\end{claim}
\begin{proof}
Note that $(\xi,\eta,h)\in P_1$ if and only if $(\eta,h)\in P$ and $\exists i\leq r$ $(\eta,\xi)\in G_{\xi_i}$ and
  $\forall j\leq r$ $Q(h,\xi_i(\eta))\subseteq
  Q(h,\xi_j(\eta))$.
Hence $P_1$ is Borel by Claims~\ref{ku} and \ref{claim:P}.
 \end{proof}

\begin{claim}\label{igrek}
The set $$L:=\{(h,\xi,\eta)\in G\times\Chb(G)\times\partial G \ | \ \textrm{$h\in Y(e,\xi)$,  $\xi\in\Xi(\eta)$}\}$$ is Borel in $G\times\Chb(G)\times\partial G$.
\end{claim}
\begin{proof}
Note that $(h,\xi,\eta)\in L$ if and only if 
$(\eta,\xi)\in M$ and $h\in\Geo(e,\eta)$ and $(\xi,\eta,h)\in P_1$ and $\dc(e,h)$ is minimal for these properties. Hence $L$ is Borel by Claims~\ref{setf}, \ref{e} and \ref{pjeden}.
\end{proof}

\begin{claim}\label{claim:B}
The set $$B:=\{(g,h,\xi,\eta)\in G^2\times\Chb(G)\times\partial G \ | \ \textrm{$g\in Q(h,\xi)$, $h\in Y(e,\xi)$, $\xi\in\Xi(\eta)$}\}$$ is Borel in $G^2\times\Chb(G)\times\partial G $.
\end{claim}
\begin{proof}
Note that $(g,h,\xi,\eta)\in B$ if and only if $\exists i\leq r$ $\xi=\xi_i(\eta)$ and $(\eta,h,g)\in Q_i$ and $(h,\xi,\eta)\in L$. Hence $B$ is Borel by Claims~\ref{ku} and \ref{igrek}.
\end{proof}

\begin{claim}\label{geo1}
The set $$A:=\{(\eta,g)\in\partial G\times G \ | \ g\in \mathrm{Geo}_1(e,\eta)\}$$ is Borel in $\partial G\times G$.
\end{claim}
\begin{proof}
 Recall that
  $$\Geo_1(e,\eta):=\bigcup_{\xi\in\Xi(\eta)}\bigcup_{h\in
    Y(e,\xi)}Q(h,\xi),$$ so that $(\eta,g)\in A$ if and only if $\exists \xi\in\Xi(\eta)$ $\exists h\in Y(e,\xi)$ $(g,h,\xi,\eta)\in B$. In other words, $A$ is obtained as the
  projection on $\partial G\times G$ of the Borel set $B$ (see Claim~\ref{claim:B}). However, vertical sections of $B$ are
  finite by Proposition
  \ref{prop:locfin_hyperb_implies_finite}. Thus, $A$ is
  Borel by the Lusin--Novikov theorem.
\end{proof}

\begin{claim}\label{claim:D}
The set $$D:=\{(\eta,(\Gamma_0,\ldots,\Gamma_n))\in\partial G\times G^{<\NN} \ | \ \textrm{$\Gamma_0\in\mathrm{Geo}_1(e,\eta)$ and $\Gamma\in\CGR(\Gamma_0,\eta)$}\}$$ is Borel in $\partial G\times G^{<\NN}$.
\end{claim}
\begin{proof}
Note that $(\eta,(\Gamma_0,\ldots,\Gamma_n))\in D$ if and only if $(\eta,\Gamma_0,(\Gamma_0,\ldots,\Gamma_n))\in F$ and $(\eta,\Gamma_0)\in A$. Hence $D$ is Borel by Claims~\ref{setf} and \ref{geo1}.
\end{proof}

Recall from Definition~\ref{definition:Ceta} the definition of the sequences $s^\eta_n$.
\begin{claim}\label{setd}
 For each $n\in\NN$, the set $$S_n:=\{(\eta,s^n)\in\partial G\times S^n \ | \ s^n=s^\eta_n \}$$ is Borel in $\partial G\times S^n$. 
\end{claim}
\begin{proof}
Note that $(\eta,s^n)\in S_n$ if and only if $\forall m\in\NN$ $\exists \Gamma^n\in G^n$ $\dc(\Gamma^n_0,e)\geq m$ and $(\eta,\Gamma^n)\in D$ and $\mathrm{typ}(\Gamma^n)=s^n$ and $s^n$ is minimal in $S^n$ for these properties. Hence $S_n$ is Borel by Claim~\ref{claim:D}.
\end{proof}

Let now $E:=\{(\eta,g\eta) \ | \ g\in G\}\subseteq \partial G\times\partial G$ denote the orbital equivalence relation for the $G$-action on $\partial G$. To prove Theorem~\ref{thmintro:mainthm}, we have to show that $E$ is hyperfinite. 

\begin{definition}
Set $Z := \{\eta\in\partial G \ | \ k^\eta_n\not\to\infty\}$. In other words, since $X$ is locally finite, $Z$ is the set of $\eta\in\partial G$ such that there exists some $g_{\eta}\in \Geo_1(e,\eta)$ belonging to $T^{\eta}_n$ for all $n\in\NN$, i.e. for which there exists some CGR $\Gamma^{\eta}\in\CGR(g_{\eta},\eta)$ of type $s^\eta=\lim_ns^\eta_n$.
\end{definition}

\begin{lemma}\label{lemma:Borel_reduction_Z}
The map $\alpha\co (Z,E|Z)\to(\partial G,=): \eta\mapsto g_{\eta}^{-1}\eta$ is a Borel reduction.
\end{lemma}
\begin{proof}
Note first that 
\begin{equation}\label{eqn:snetatheta}
s^\eta_n=s^{g\eta}_n\quad\textrm{for all $g\in G$ and $\eta\in \partial G$.}
\end{equation}
Indeed, if there are infinitely many couples $(h,s^\eta_n)\in C^\eta$, then there are infinitely many couples $(gh, s^\eta_n)\in G\times S^n$ with $s^\eta_n$ the type of the initial segment of length $n$ of a CGR $\Gamma\in\CGR(\Gamma_0,g\eta)$ with $\Gamma_0\in g\Geo_1(e,\eta)=\Geo_1(g,g\eta)$ (see Lemma~\ref{lemma:Geo1_G_equiv}), and hence there are infintely many couples $(gh, s^\eta_n)$ in $C^{g\eta}$ as the symmetric difference between $\Geo_1(g,g\eta)$ and $\Geo_1(e,g\eta)$ is finite by Theorem~\ref{theorem:FSD}.

In particular, $\alpha$ is constant on $G$-orbits, as it maps the class of $\Gamma^\eta$ to the class of $g_{\eta}^{-1}\Gamma^{\eta}$, and $g_{\eta}^{-1}\Gamma^{\eta}$ is the unique CGR $\Gamma_{e,s^\eta}$ from $e$ of type $s^\eta=s^{g_{\eta}^{-1}\eta}$. This shows that $\alpha$ is a reduction from $E|Z$ to the identity, and it remains to see that $\alpha $ is Borel. But this boils down to the Borelness of the set $\{(\eta,s)\in \partial G\times S^{\NN} \ | \ s=s^\eta\}$, which follows from Claim~\ref{setd} as $s=s^\eta$ if and only if $\forall n\in\NN$ $(\eta,(s_0,\dots,s_{n}))\in S_{n+1}$.
\end{proof}

\begin{lemma}\label{lemma:EZsmooth}
$E$ is smooth on the $E$-saturation $[Z]_E:=\{\eta\in\partial G \ | \ \exists \theta\in Z \ \theta E \eta\}$ of $Z$.
\end{lemma}
\begin{proof}
Lemma~\ref{lemma:Borel_reduction_Z} implies that $E$ is smooth on $Z$, whence the claim.
\end{proof}

It now remains to check that $E$ is hyperfinite on the complement of $[Z]_E$ in $\partial G$.

\begin{definition}
Let $Y$ be the complement of $[Z]_E$ in $\partial G$, i.e. $$Y = \{\eta\in\partial G \ | \ \forall
\theta\in\partial G: \ \theta E \eta\implies k^\theta_n\to\infty\}.$$
For each $n\in\N$,
we define the function $H_n:\partial G\to 2^G$
by $$H_n(\eta) := (g^\eta_n)^{-1} T^\eta_n.$$
Let $F_n$ be the equivalence relation on $\im H_n$
which is the restriction of the shift action of $G$ on $2^G$.
\end{definition}

The following lemma has exactly the same proof as \cite[Lemma 5.2]{HSS17} and we reproduce it here for completeness.

\begin{lemma}\label{lemma:cstK}
  There exists a constant $K>0$ such that for each $n\in\NN$, the relation $F_n$ on $\im H_n$ has equivalence classes of size at most $K$.
\end{lemma}
\begin{proof}
We will prove the lemma for $K$ the size of the ball centred at $e$ and of radius $8\delta$ in $X$. Let $n\in\NN$.
   Let $\eta,\theta\in\partial G$ and suppose $H_n(\eta) = gH_n(\theta)$ are
  $F_n$-related as witnessed by $g\in G$. Then $(g_n^\eta)^{-1}T_n^\eta=g(g_n^\theta)^{-1}T_n^\theta$. 
  We will show that $\dc(e,g)\leq 8\delta$, as desired.

Note that since $T^\eta_n$ (resp. $T^\theta_n$) is an infinite subset
  of $\Geo(e,\eta)$ (resp. $\Geo(e,\theta)$), it uniquely determines the
  boundary point $\eta$ (resp. $\theta$). In particular,
   $$\sigma:=(g^\eta_n)^{-1} \eta = g(g^\theta_n)^{-1} \theta.$$

\[
  \begin{tikzpicture}
    \coordinate (a) at (0,-3);
    \coordinate (b) at (0,3);
    \coordinate (m4) at (1,0.3);
    \coordinate (v) at (2.5,1);
    \coordinate (gv) at (2,-1.3);
    \coordinate (m2) at (4,0);
    \coordinate (m5) at (4,-2);
    \coordinate (eta) at (7,0);
    \coordinate (gamma) at (7,-2);
    \coordinate (c) at (8,0);
    \filldraw (a) circle (1mm) node[yshift=-5mm] { $\boldsymbol{g(g^\theta_n)^{-1}}$};
    \filldraw (b) circle (1mm) node[right] { $\boldsymbol{(g^\eta_n)^{-1}}$};
    \filldraw (m4) circle (1mm) node[left] {$\boldsymbol{\Gamma_{m_4}}$};
    \filldraw (v) circle (1mm) node[above right] { $\boldsymbol{e}$};
    \filldraw (gv) circle (1mm) node[yshift=-4mm] {$\boldsymbol{g}$};
    \filldraw (m2) circle (1mm) node[below right] {$\boldsymbol{\Lambda_{m_2}}$};
    \filldraw (m5) circle (1mm) node[below] {$\boldsymbol{\Gamma_{m_5}}$};
    \draw (eta) node[right] {$\boldsymbol{\Lambda}$};
    \draw (gamma) node[right] { $\boldsymbol{\Gamma}$};
    \draw [ultra thick, ->] plot [smooth] coordinates {(a) (gv) (m2) (eta)};
    \draw [ultra thick, ->] plot [smooth] coordinates {(b) (m4) (gv) (m5) (gamma)};
    \draw [ultra thick, dashed] (m4) -- (v) -- (m2) -- (m5);
  \end{tikzpicture}
  \]

As $g,e\in (g^\eta_n)^{-1}T^\eta_n=g(g^\theta_n)^{-1}T_n^{\theta}\subseteq \Geo(g(g^\theta_n)^{-1},\sigma)$, there exist a CGR $\Lambda\in\CGR(g(g^\theta_n)^{-1},\sigma)$ passing through $g$, say $g=\Lambda_{m_1}$ for some $m_1\in\NN$, and a CGR $\Lambda'\in\CGR(g(g^\theta_n)^{-1},\sigma)$ passing through $e$, say $e=\Lambda'_{m_2}$ for some $m_2\in\NN$. Note that $\dc(e,\Lambda_{m_2})\leq 2\delta$ by Lemma~\ref{lemma:hyperbolic_basic_prop}. Moreover, $m_2\geq m_1$: indeed, as $g_n^\theta g^{-1}\in g_n^\theta g^{-1}(g_n^\eta)^{-1}T_n^\eta=T_n^\theta$, the minimality of $g_n^\theta$ in $T_n^\theta$ implies that 
$$m_2=\dc(e,g(g^\theta_n)^{-1})=\dc(e,g_n^\theta g^{-1})\geq \dc(e,g_n^\theta)=\dc(e,(g_n^\theta)^{-1})=\dc(g,g(g_n^\theta)^{-1})=m_1.$$ 

Similarly, $g,e\in \Geo((g_n^\eta)^{-1},\sigma)$, and hence there exist $\Gamma\in \CGR((g_n^\eta)^{-1},\sigma)$ passing through $g$, say $g=\Gamma_{m_3}$ for some $m_3\in\NN$, and $\Gamma'\in \CGR((g_n^\eta)^{-1},\sigma)$ passing through $e$, say $e=\Gamma'_{m_4}$ for some $m_4\in\NN$. Again, $\dc(e,\Gamma_{m_4})\leq 2\delta$ by Lemma~\ref{lemma:hyperbolic_basic_prop}, and $m_4\leq m_3$ because $g_n^\eta g\in g_n^\eta g(g^\theta_n)^{-1}T_n^{\theta}=T^\eta_n$ and hence the minimality of $g_n^\eta$ in $T_n^\eta$ implies that
$$m_3=\dc((g_n^\eta)^{-1},g)=\dc(e,g_n^\eta g)\geq \dc(e,g_n^\eta)=\dc(e,(g_n^\eta)^{-1})=m_4.$$ 

 Applying Lemma~\ref{lemma:hyperbolic_basic_prop} one more time to the sub-CGR of $\Gamma$ and $\Lambda$ starting at $g$, and using the fact that $m_2\geq m_1$, we find some $m_5\ge m_3$ such that
  $\dc(\Lambda_{m_2},\Gamma_{m_5})\le 2\delta$.
  
  Note that
  $$\dc(\Gamma_{m_4},\Gamma_{m_5})\leq
  \dc(\Gamma_{m_4},e)+\dc(e,\Lambda_{m_2})+\dc(\Lambda_{m_2},\Gamma_{m_5})\leq
  6\delta.$$ Thus,
  $$\dc(e,g)=\dc(e,\Gamma_{m_3})\leq
  \dc(e,\Gamma_{m_4})+\dc(\Gamma_{m_4},\Gamma_{m_3})\leq
  2\delta +\dc(\Gamma_{m_4},\Gamma_{m_5})\leq 8\delta,$$
  where the inequality $\dc(\Gamma_{m_4},\Gamma_{m_3})\leq
  \dc(\Gamma_{m_4},\Gamma_{m_5})$ follows from $m_5\geq m_3\geq m_4$.
\end{proof}

\begin{lemma}\label{lemma:HnBorel}
 Let $n\in\NN$. Then the map $H_n$ is Borel. In particular, the set $\im H_n$ is analytic.
\end{lemma}
\begin{proof}
The sets $S_n=\{(\eta,s^\eta_n)\in\partial G\times S^n\}$, $\{(\eta,
 g^\eta_n)\in\partial G\times G\}$, $\{(\eta, T^\eta_n)\in\partial G\times 2^G\}$, and $G_{H_n}=\{(\eta,H_n(\eta))\in\partial G\times 2^G\}$ are easily seen to be definable using only formulas with countable quantifiers and references to the Borel set $D$ (see Claims~\ref{claim:D} and \ref{setd}). In particular, these sets are Borel.
Hence $H_n$ is Borel (as its graph $G_{H_n}$ is Borel) and $\im H_n=\proj_{2^G}(G_{H_n})$ is analytic, as desired.
\end{proof}

By Lemma~\ref{reflection} (applied to $Z:=2^G$, $A:=\im H_n$, $E$ the shift action of $G$ on $2^G$, and $K$ the constant from Lemma~\ref{lemma:cstK}),
we find a finite Borel equivalence relation $F_n'$ on $2^G$
with $F_n\subseteq F_n'$. Let $f_n:2^G\to 2^\N$ be a
reduction from $F_n'$ to $E_0$ (such a reduction exists as every finite Borel
equivalence relation is smooth)
and define $f:\partial G\to (2^\N)^\N$ by $f(\eta) =
(f_n(H_n(\eta)))_{n\in\N}$. 
Let $E':=f^{-1}(E_1)$ (i.e. $\eta \mathrel{E'}\theta\iff f(\eta) \mathrel{E_1}
f(\theta)$ for all $\eta,\theta\in\partial G$). 

\begin{lemma}\label{lemma:Ephyp}
The relation $E'$ is hyperfinite.
\end{lemma}
\begin{proof}
Note first that $E'$ is Borel, because each $H_n$ is Borel by Lemma~\ref{lemma:HnBorel}. Note next that
$E'$ is a countable equivalence relation. To see this,
let $E_n'$ ($n\in\NN$) be the relation on $\partial G$ defined by $\eta\mathrel{E_n'}\theta$ if
$f_m(H_m(\eta))=f_m(H_m(\theta))$ for all $m\geq n$. Note that each
$E_n'$ is a countable equivalence relation. Indeed, if
$\eta\mathrel{E_n'}\theta$, then in particular
$f_n(H_n(\eta))=f_n(H_n(\theta))$. However, the function
$f_n\circ H_n$ is countable-to-one as $H_n$ is countable-to-one (namely, if $H_n(\eta)=H_n(\theta)$ then $(g^\eta_n)^{-1} \eta = (g^\theta_n)^{-1} \theta$, and hence $\eta\in G\cdot\theta$) and $f_n$ is
finite-to-one, yielding the claim. Thus, $E'=\bigcup_{n\in\N} E'_n$ is indeed
countable.  Also, $E'$ is by definition hypersmooth. Thus, by \cite[Theorem 8.1.5]{gao},
the relation $E'$ is hyperfinite.
\end{proof}

\begin{lemma}\label{lemma:EYsubE1}
  The function $f$ is a homomorphism from $E|Y$ to $E_1$.
\end{lemma}
\begin{proof}
Suppose $\eta,\theta\in Y$ are such that
$g\eta=\theta$ for some $g\in G$. 
By Theorem~\ref{theorem:FSD} (and Lemma~\ref{lemma:Geo1_G_equiv}), the sets $g\mathrm{Geo}_1(e,\eta)$ and $\mathrm{Geo}_1(e,\theta)$
differ by a finite set. Since both $\eta,\theta$ are
in $Y$ and $X$ is locally finite,
there is some $N\in\N$ such that $g T_n^\eta\subseteq\Geo_1(e,\theta)$ for all $n\geq N$. Since $s_n^\eta=s_n^\theta$ by (\ref{eqn:snetatheta}), we then have $gT_n^\eta=T_n^\theta$ and hence $(g_n^\theta)^{-1}gg_n^{\eta}H_n^\eta=H_n^\theta$ for all $n\geq N$. Thus $H_n(\eta) \mathrel{F_n} H_n(\theta)$, and hence $H_n(\eta) \mathrel{F'_n} H_n(\theta)$ for all $n\geq N$. Therefore, $f_n(H_n(\eta))=f_n(H_n(\theta))$ for all $n\geq N$, that is, $f(\eta)\mathrel{E_1} f(\theta)$.
\end{proof}

\noindent
{\bf Proof of Theorem~\ref{thmintro:mainthm}:}
By Lemma~\ref{lemma:EYsubE1}, the relation $E|Y$ is a subrelation of
$E'$. A subrelation of a hyperfinite equivalence relation is
also hyperfinite, so $E|Y$ is hyperfinite by Lemma~\ref{lemma:Ephyp}. On the other hand, $E$ is smooth, and hence also hyperfinite, on $\partial G\setminus Y=[Z]_E$ by Lemma~\ref{lemma:EZsmooth}. Therefore, $E$ is hyperfinite on $\partial G$, thus concluding the proof of Theorem~\ref{thmintro:mainthm}. $\hspace{\fill}\qed$

\bibliographystyle{amsalpha} 
\bibliography{biblio_horohyp} 

\end{document}